%% file: main_7.tex
\documentclass[pdftex,12pt,a4paper]{amsart}
\usepackage[pdftex]{graphicx}
\DeclareGraphicsExtensions{.png,.pdf}
\usepackage{thumbpdf}
\usepackage{comment}
\usepackage[export]{adjustbox}
\usepackage{mathrsfs}
\graphicspath{{figs/}}
\usepackage[english]{babel}
\usepackage[utf8]{inputenc}
\usepackage{amssymb}
\usepackage{amsmath}
\usepackage{amsthm}
\usepackage{color}
\usepackage{hyperref}
\usepackage{orcidlink}
\usepackage{url}
\usepackage{bm}
\usepackage{enumerate}
\usepackage{tikz}
\usepackage{tikz-cd}
\usepackage{algorithmic}
\usepackage{algorithm}
\usepackage{longtable}
\usepackage{multirow}
\usepackage{microtype}
\usepackage{geometry}
\input{macros}

\newcommand{\mybox}[1]{\raisebox{0pt}[2.2ex][0ex]{$#1$}}
\newcommand{\tol}{\mathop{\rm tol}\nolimits}
\makeatletter
\newcommand{\subalign}[1]{%
  \vcenter{%
    \Let@ \restore@math@cr \default@tag
    \baselineskip\fontdimen10 \scriptfont\tw@
    \advance\baselineskip\fontdimen12 \scriptfont\tw@
    \lineskip\thr@@\fontdimen8 \scriptfont\thr@@
    \lineskiplimit\lineskip
    \ialign{\hfil$\m@th\scriptstyle##$&$\m@th\scriptstyle{}##$\hfil\crcr
      #1\crcr
    }%
  }%
}
\makeatother

\theoremstyle{definition}
\newtheorem{Def}{Definition}[section]
\newtheorem{remark}[Def]{Remark}

\theoremstyle{plain}
\newtheorem{proposition}{Proposition}[section]

\theoremstyle{remark}

\counterwithin{algorithm}{section}

\algsetup{linenodelimiter=)}

\begin{document}

\title[Computation of  invariant tori]{
Simultaneous computation of 
whiskered tori and their whiskers 
in Hamiltonian systems using flow maps
}
% Flow map parameterization methods for whiskered tori in Hamiltonian systems }

\author[\'A. Fernández-Mora]{{\'A}lvaro
Fern{\'a}ndez-Mora$^{\mbox{a}}$ \orcidlink{0000-0002-3863-0003} }
\address[a]{Departament de Matem\`atiques i Inform\`atica,
Universitat de Barcelona, Gran Via 585, 08007 Barcelona, Spain.}
\email{alvaro@maia.ub.es}

\author[\`A. Haro]{{\`A}lex Haro$^{\mbox{a,b}}$ \orcidlink{0000-0003-0377-8099}}
\address[b]{Centre de Recerca Matem\`atica, Edifici C, Campus
Bellaterra, 08193 Bellaterra, Spain}
\email{alex@maia.ub.es}

\author[R. de la Llave]{Rafael de la Llave$^{\mbox{c}}$ \orcidlink{0000-0002-0286-6233}}
\address[c]{School of Mathematics (emeritus), Georgia Institute of
Technology, 686 Cherry Street NW, Atlanta, GA 30332}
\email{rafael.delallave@math.gatech.edu}

\author[J. M. Mondelo]{Josep M. Mondelo$^{\mbox{d}}$\orcidlink{0000-0002-7135-0599}}
\address[d]{IEEEC-CERES \& Departament de Matem\`atiques, Universitat Aut\`onoma
de Barcelona, Av.~de l'Eix Central, Edifici C, 08193
Bellaterra (Barcelona), Spain.}
\email{josemaria.mondelo@uab.cat}

\thanks{This work has been supported by the European Union's
   Horizon 2020 research and innovation program under the Marie
   Sk{\l}odowska-Curie grant agreement 734557, the MCINN-AEI grants
   PID2021-125535NB-I00 and PID2020-118281GB-C31, the Catalan
   grant 2021 SGR 00113, the Severo Ochoa and Mar\'{\i}a de
   Maeztu Program for Centers and Units of Excellence in R\&D
   (CEX2020-001084-M) of the Spanish Research Agency, the
   Secretariat for Universities and Research of the Ministry of
   Business and Knowledge of the Government of Catalonia, and by
the European Social Fund.}

\date{}

\begin{abstract}
We consider autonomous Hamiltonian systems 
and present an algorithm to compute \emph{at  the same time} 
 partially
hyperbolic invariant tori (whiskered tori), as well as 
high-order expansions of their stable and unstable manifolds. 
Such whiskered tori have been shown to be important for transport
phenomena in phase space. For instance, by following their
invariant manifolds one could obtain zero-cost trajectories in
space mission design. We present in detail 
the case when the (un)stable directions are one-dimensional.  

The strategy to compute tori and their invariant manifolds
is based on the parameterization method. We formulate a
functional equation for a parameterization of both the torus  and  its 
whiskers
expressing that they are invariant.  This equation is naturally
discretized in Fourier-Taylor series or, equivalently, in a grid
of Taylor series. Using a return map, we are reduced to study 
functions of $n - 1$ variables where $n$ is the number of degrees
of freedom (the phase space is $2n$ dimensional). Then, we
implement a Newton-like method that converges quadratically.

They key advantage  of our approach is that, using geometric
identities coming from the Hamiltonian nature of the
problem, the algorithm has small storage requirements and
a low operation count per step which is highly efficient.
The simultaneous computation of  the torus and the whiskers 
improves  the efficiency and the stability of the algorithm.  

We present implementations and extensive numerical experiments in
the Circular Restricted Three Body Problem.
\end{abstract}

\maketitle

\section{Introduction}
Both in theory and applications, it is of great interest to
identify the asymptotic behavior of orbits of dynamical
systems---specially when the orbits approach invariant sets.  Such
asymptotic convergence  is described by the stable and unstable
manifolds of hyperbolic objects, e.g., fixed points, periodic
orbits, or whiskered tori. A torus is said to be whiskered if its
motion is a rotation and, associated to each of its points, there exists 
directions that contract or expand exponentially under the
linearization of the dynamics. These are the stable and unstable bundles.
More importantly (as a mathematical consequence of 
the existence of contracting bundles) there are stable and
unstable manifolds that contract or expand exponentially under
the full dynamics.  The intersection of these manifolds can lead
to interesting dynamics.

In modern times, \cite{Arnold64, ArnoldA68} showed that
one can generate large scale motions in phase space
by having chains of whiskered tori so that the unstable
manifold of one intersects the stable manifold of the next.
Such sequences are called transition chains. 
In \cite{Arnold63} it is conjectured with good evidence,
that transition chains happen often in celestial mechanics and
that this phenomena is in fact the main mechanism for diffusive
orbits. Even if other mechanisms for global instability have
been discovered, whiskered tori---specially those generated 
by resonances \cite{Treschev89}---provide very complicated maneuvers. 
Rigorous proofs of existence of whiskered tori 
and several or their properties (including the existence of 
(un)stable manifolds) were established 
in \cite{Graff74,Zehnder76}. 

The pioneering works \cite{ESA1,ESA2,ESA3,ESA4} marked
significant progress in the application of advanced numerical
dynamical systems techniques to celestial mechanics.
These books focus on libration point missions for which the
Restricted Three Body problem is the basic model. Dynamical
connections between hyperbolic objects serve as zero-fuel
trajectories.\footnote{Some small amount of fuel may be  needed to
make small corrections due to measurement errors or solar
winds and such.} For instance, the Genesis \cite{Genesis} and
Artemis \cite{Artemis} space missions used trajectories close to
heteroclinic connections for the design of their itineraries.
There are already several applied papers that have shown how to
compute connections between whiskered tori \cite{OwenB24, HenryS23,
Anderson21,McCarthyH23, CallejaDHLO12, BonaseraB20}. We hope
that the present work could help to provide more extensive catalogues  of
whiskered tori to be used as mileposts in diffusion routes or in
zero-cost maneuvers in mission design. 

A powerful approach to compute invariant objects and to prove
their existence is the parameterization method. Following this
framework, invariant tori and their
whiskers were computed in \cite{HaroL06a,HaroL06b,HaroL07} for
quasi-periodically forced maps using a variety of methods. Under
the same paradigm, \cite{FontichLS09} obtains {\em a-posteriori}
theorems for the existence of whiskered tori in exact symplectic
maps and exact symplectic vector fields. The proof in
\cite{FontichLS09} applies even when the (un)stable bundles are
non-trivial and their a-posteriori theorems show that the approximate
numerical computations obtained here are close---in an explicit
sense---to true solutions.  Relatedly, fast iterative schemes
were described in \cite{HuguetLS12} for both Lagrangian and
whiskered tori in Hamiltonian systems whereas in
\cite{FMHM24,FMHMKAM}, the authors provide theorems, algorithms,
and numerical results for whiskered tori and their invariant
bundles in quasi-periodic Hamiltonian systems. For resonant tori
see \cite{KumarAL21} and for a different approach based on
quasi-periodic Floquet transformations for response tori see
\cite{Gimeno22}.  The invariant bundles and whiskers can then
subsequently be used to compute dynamical connections.

For instance, numerical techniques were developed in
\cite{2009BaMoOll,BarrabesMO13} to compute families of
heteroclinic and homoclinic connections
between periodic orbits in Hamiltonian systems, whereas
computer-assisted methods for individual connections were used 
in \cite{WilczakZ03,WilczakZ05}. Naturally, dynamical connections
between higher dimensional objects also exist. See for instance
\cite{GomezKLMMR04, CanaliasM08} for the Circular Restricted
Three Body Problem (CRTBP) and \cite{KumarAL21,Kumar21GPU} for
connections between resonant tori in the planar circular and
elliptic problems, respectively. For entire families of
connections between center manifolds see \cite{Barcelona24}.
Other mathematical studies on the generation of intersections are
\cite{Eliasson94}. 
The paper \cite{FontichM00} showed that indeed, given a
transition chain (even of infinite length), there are orbits of
the system that follow it.  The classical results on existence of
complicated trajectories for traversal homoclinic intersections
of fixed points \cite{Poincare87c, Smale63,Smale67} extend also
to homoclinics of whiskered tori \cite{Bolotin95, Bolotin90,
GideaR03}. An area of current theoretical and practical interest
is the computation of heteroclinic connections among tori of
different dimensions.

One of the motivations of this work is to enable the construction
of detailed maps of whiskered tori and their invariant manifolds
in regions of interest. This is based on developing 
very efficient algorithms. With this goal in mind, we are interested
in computing accurately not only tori and their whiskers
but also the tangents. This makes it possible to
develop Newton methods to compute connections between invariant
tori or to deploy control algorithms in their neighborhood. 
We consider only tori for which the stable/unstable directions
are 1-dimensional. This  is the case most often considered in
celestial mechanics as in \cite{Arnold63}. Of course, we hope to
return to this problem and develop methods to compute invariant
tori with higher dimensional stable/unstable manifolds as well as
methods for the computation of whiskers in quasi-periodic
systems. Such methods can be useful in more realistic models in
celestial mechanics such as those of \cite{gomez2002solar} or
within the framework of \cite{UTSHowell25}.

The goal of this paper is then to develop and implement a method
for the simultaneous computation of parameterizations of
invariant tori and their whiskers by means of
Kolmogorov-Arnold-Moser schemes. In this work, we only describe
the algorithms and  implementations. We postpone for future
work theoretical studies of convergence and a-posteriori
theorems. 

Following \cite{HM21,FMHM24,FMHMKAM}, we make use of time-$T$ map
reductions, which allows one to work with functions of one fewer
variable than when working directly with the vector field.  This
improves the efficiency of the calculations since it mitigates
the \emph{curse of dimensionality}. It is also natural to
separate the calculation along the flow and the transversal
directions. One reason is that the torus remains very
differentiable along the flow and the algorithms for numerical
integration have been studied for several centuries and are well
understood. Another reason to separate the propagation along the
flow from other computations is that integration is very
parallelizable.

The method described here is based on a Newton-like iteration 
that computes simultaneously high order Fourier-Taylor expansions
for tori and their whiskers (as well as the normal contraction). 
The simultaneous computation leads to some acceleration because
improving the invariance equation for whiskers improves the
computation of tori. Compared with methods that compute
first the torus and then the whiskers
order by order, see e.g. \cite{KumarALl22},
we note that each iteration of our method---if the
torus is computed exactly---doubles the order the whiskers
are computed exactly. The numerical errors at all orders are
corrected at each step of the iterative procedure. 

We use identities coming from Hamiltonian geometry to simplify
the iterative step into substeps that are diagonal either in
Fourier representation or in grid representation. The result is
a highly efficient algorithm that does not require tuning.
Another improvement that we have
implemented is that, to take care of instabilities in the
propagation, we set up a multiple shooting approach.

The method we present is extremely efficient.
It is quadratically convergent as a Newton method
but does not need to store a matrix (much less to invert it). 
If we discretize the torus and whiskers using $N$ coefficients,
the algorithm only requires $\cO(N)$ storage and a step
only requires $\cO(N \log N)$ operations.
This is to be contrasted with conventional Newton methods (see,
e.g., \cite{Olikara16, GomezM01}), which require $\cO(N^2)$ storage and
$\cO(N^3)$ operations, as is the case for approaches that
discretize the invariance equation using collocation points or
Fourier series and apply a standard Newton method.
The paper \cite{HaroL07} contains a comparison between 
large matrix methods and fast methods based on reducibility
(both reducibility due to geometric identities as here 
and reducibility obtained through numerical computation).

In summary, the method presented here satisfies:
\begin{itemize}
\item  
The method is quadratically  convergent as a Newton method. That is, 
the error after applying one step is comparable to the square of 
the error before starting. 
\item
  In a system with $n$ degrees of freedom, we need to deal
  with functions of $n-1$ variables.
  Note that the cost of representing a function grows
  exponentially with the number of variables but only linearly
  with the dimension of the range. 
  \item 
  The storage requirements are small (no matrices needed).
  If parameterizations are discretized in $N$ Fourier-Taylor or grid-Taylor
  coefficients, the storage needed is $\cO(N)$.
  \item
  The operation count for one step of the iterative procedure is
  $\cO(N \log N)$.
  The iterative step breaks into substeps which 
  require only $\cO(N)$ operations (either in Fourier representation
  or in grid representation). Then, the $\cO(N \log N)$ operation
  count comes from switching representations with the FFT.
  In practice the FFT is well optimized even in hardware, so that
  it takes even less time than the theoretical bound. 
\item
  Its implementation is straightforward. The algorithm for one iterative step
  breaks into 11 steps, see Algorithm \ref{algo:newton}, each of
  which consists of natural operations and require few
  computational parameters. Many of these steps can be
  implemented in a single instruction in modern programming
  languages. 
\item
  It is backed up by \emph{a-posteriori} theorems which
  show it is reliable provided some (explicit) non-degeneracy
  conditions are met.

\item
  The algorithm obtains information not only of the invariant
  objects but also of their derivatives, making it possible to
  develop Newton methods to compute intersections or to devise
  control methods for space mission design. Such developments are
  not carried out in this paper. 
\end{itemize}

The basic idea for the fast Newton-like iteration is that we use
geometric identities to devise a method that only involves
operations in objects that are of size $\cO(N)$. We observe that
the geometric properties of the map imply that there is an
explicit frame where the linearization of the motion on the
whisker is an upper triangular matrix. This has been used for
invariant tori, but we now observe that this applies also for 
whiskers.  We note, however that methods based on parameterization 
methods (that are not as fast as those presented here)
but that work without leveraging geometry also exist 
\cite{CanadellH17a,CanadellH17b}. 

Once a local representation of the whiskers is computed, one can
use the flow in the future and in the past to get a global manifold.
Given this goal, it is worth to obtain an accurate local
representation in a {\em fundamental domain} which
we want to be as large as possible. The algorithm we present has
few computational parameters. Mainly the
number of coefficients used and a scale parameter whose theory
is studied in Section~\ref{sec:iter}, which affects the
numerical stability of the computations. In our numerical
runs, we typically make a preliminary calculation that allows to
fine tune the parameters and then, we make longer runs with the
optimized parameters. The algorithm can run largely unattended
in a variety of problems. 

We would like to note that a Newton method leads immediately to
continuation strategies on any parameter. Take the solution at a
value of a parameter as the starting point for the Newton method
for a slightly different value. Here, we present details about a
particularly efficient continuation method (including
extrapolations) along the period. This is equivalent to
conducting continuations along the energy. In the study of
continuation of solutions, we point out that the a-posteriori
result shows that the continuation can proceed until the torus
fails to satisfy some of the geometric conditions of the theorem,
which indicates breakdown.

Regarding implementations, we have applied our techniques in the
CRTBP where we computed a large set of non-resonant
$2-$dimensional invariant tori and their stable/unstable
manifolds studying the fundamental domain of the
parameterizations.  As expected, the methods are
extremely fast and for a fairly large number of coefficients, 
the calculations for realistic models such as 
the CRTBP are in the order of seconds in a commercial machine of
today without any sophisticated fine tuning. See Table
\ref{tab:comp-times}. 

This paper is organized as follows: In Section
\ref{sec:setting}, we introduce the setting and provide the
necessary background material. Section \ref{sec:invtori}
discusses whiskered tori and their geometric properties which
enable the construction of adapted frames.
In Section \ref{sec:Newton}, we present the iterative
procedure for computing invariant tori and their whiskers,
propose strategies for their continuation with respect
to parameters, and discuss how to study the fundamental domain of
the parameterizations. The algorithms are described in
pseudo-code. Section \ref{sec:considerations} addresses
computational aspects of the methods of Section \ref{sec:Newton}.
Lastly, in Section \ref{sec:numerics}, we apply our techniques to
the CRTBP and report on the results.

\section{Setting}
\label{sec:setting}
We assume all objects to be sufficiently
smooth---even real analytic. In this paper we
concentrate on the algorithms and postpone precise formulations
of regularities (which are essential to study convergence using
KAM theory) to a theoretical paper. 
 
\subsection{Geometric structures}
We will be working in an Euclidean space (more suitable for numerical
analysis). All the geometric objects will be expressed as matrices
and the geometric identities used become matrix identities. 

We assume we have an exact symplectic form
$\bm\om=\dif\bm\al$ on an open set $U\subset\RR^{2n}$,
where $\bm \al$ is the Liouville 1-form, endowing $ U $ with an
exact symplectic structure. Let $\Omega: U \to\RR^{2n\times2n}$
be the matrix representation of $\bm\om$. Since $\bOm$ is exact,
the matrix representation of the symplectic form satisfies
\begin{equation*}
\Omega(z)=\Dif a(z)^\ttop - \Dif a(z),
\end{equation*}
where $a(z)^\ttop$ is the matrix representation of $\bm\al$ at
$z$.
Then, for any $z\in U $ and $u,v\in \RR^{2n}$
\begin{gather*}
   \bm\om_z(u,v) = u^\ttop\Omega(z)v, \\
   \bm\al_z(u) =a(z)^\ttop u = u^\ttop a(z).
\end{gather*}

Let $\bJ$ be an almost complex structure on $U$
compatible with the symplectic form and with matrix representation 
$J:U\to\RR^{2n\times 2n}$. That is, $J$ is
anti-involutive and symplectic. In coordinates, these properties
read
\[
   J(z)^2=-I_{2n},\quad
J(z)^\ttop\Om(z)J(z)=\Om(z).
\]

The almost complex structure and the symplectic form induce a
Riemannian metric $\bg$ on $ U $ with matrix representation
$G: U \to\RR^{2n\times 2n}$ constructed as $G(z)=-\Om(z)J(z)$.

It is well known that one can arrange to have 
a compatible triple of symplectic form, almost complex 
structure, and metric $(\bOm,\bJ,\bg)$ such that 
for all $u,v\in \RR^{2n}$ and $z\in U $, the compatible 
triple satisfies
\[
\bOm_z(u,v) = \bm{g}_z(u,\bJ_z v).
\]
Note that any two elements of the triple determine 
the other one \cite{Berndt01}. 
Generally, the metric $\bm{g}$ corresponding to 
the two previously made choices is not the standard metric.
Nonetheless, for the numerical explorations, we use the standard
symplectic form and the standard metric.

The use of an almost complex structure is not strictly
necessary but it leads to significant simplifications and
numerical advantages.  Some of the properties that the matrix
representations satisfy are the following 
\begin{align*}
   \Omega(z)& = G(z)J(z), & \ \Omega(z)^\ttop &= -\Omega(z), \\
   J(z)^\ttop&= -J(z), & G(z)^\ttop &= G(z).
\end{align*}

For any smooth function $H: U \to\RR$, we construct a
Hamiltonian vector field $X: U \to \RR^{2n}$ as
\begin{equation}\label{eq:VF}
   X(z):=\Omega(z)^{-1} \Dif H(z)^\ttop.
\end{equation}
Let $\phi:\cD\subset\RR\times U\to U $ denote the flow for the
vector field $X$.  We will use the standard notation
$\phi(t,z)=\phi_t(z)$.  For fixed $t$, the diffeomorphism
$\phi_t$ is exact symplectic---that is, $\phi_t$ preserves the
symplectic form as
\[
\phi_t^*\bm \om = \bm\om,
\]
and the 1-form $\phi_t^*\bm \al-\bm\al$ is exact for some smooth
primitive function, see \cite{Zehnder76,HM21}.

\begin{remark}
It is possible to adapt this setting to other contexts, for
instance, where the phase space is some annulus
$(\TT^\ell\times\RR^{n-\ell})\times\RR^n$.
\end{remark}

\subsection{Fourier-Taylor series}\label{sec:FT}
In this work, we encounter functions defined in $\TT^d$ and
$\TT^d\times\RR$. We dedicate this section 
to describing  the representations of functions we will consider.

For every $\xi:\TT^{d}\times\RR\to \RR$, let its 
Fourier-Taylor series be given by the expansion
\begin{equation}\label{eq:FT}
\xi(\te,s) = \sum_{j=0}^\infty \xi_j(\te)s^j.
\end{equation}
Every $\xi_j:\TT^d\to\RR$ is 1-periodic in each component of
$\te\in\TT^d$ and can be represented by its Fourier series
\begin{equation} \label{eq:Fseries}
\xi_j(\te)=\sum_{k\in\ZZ^{d}} \hat\xi_{jk} \ee^{\bi 2\pi k\te},
\end{equation}
where $\hat\xi_{jk}\in\CC$ and $k \te:=k_1\te^1+...+k_{d}\te^{d}$.
The average of $\xi_j$ on $\TT^{d}$ is the coefficient
$\hat\xi_{j0}$ that we will denote as
\[
\langle\xi_j\rangle:=\int_{\TT^{d}}\xi_j(\te) d\te.
\]

\subsection{Cohomological equations} \label{sec:coho}
Given 
$\al,\be\in\RR,~\om\footnote{ Note that $\om$ is a rotation vector.  The
symplectic form is 
$\bm \om$.}\in\RR^d\setminus\{0\},~\la\in\RR\setminus
\{\pm 1,0\},$
and a real-analytic function $\eta:\TT^d\to\RR$,  
we will encounter the following functional equations for $\xi:\TT^d\to\RR$
\begin{equation}\label{eq:cohoFT}
 \al\xi(\te,s)-\be\xi(\te+\om,\la s) = \eta(\te,s),
\end{equation}
where we will require $\om$ to be Diophantine.
\begin{Def}\label{def:diophantine}
The rotation vector $\om\in\RR^{d}$ is Diophantine if
there exists $\gamma>0$ and $\tau\geq d$
such that for all $n\in\ZZ$ and $k\in\ZZ^{d}\setminus \{ 0\}$
\[
   |k\cdot\omega-n|\geq
   \frac{\gamma}{|k|_1^\tau},
\]
where $|\cdot|_1$ is the $\ell^1$-norm.
\end{Def}

If we use the Fourier-Taylor series of $\xi$ and $\eta$ as in 
\eqref{eq:FT}, we obtain---for each $j$---the
cohomological equation
\begin{equation} \label{eq:coho}
\al\xi_j(\te)-\be\la^j\xi_j(\te+\om) = \eta_j(\te).
\end{equation}
Then if we express $\xi_j$ and $\eta_j$ as Fourier series,
the coefficients of the solution for \eqref{eq:coho}
are formally given by 
\begin{equation}\label{eq:cohosol}
\hat \xi_{jk} = \frac{\hat\eta_{jk}}{\al-\be\la^j e^{\bi 2\pi
k\om}}.
\end{equation}
Observe that when $\al=\be\la^j$, the equation~\eqref{eq:coho} 
becomes the standard small divisors difference equation
\[
\xi_j(\te) - \xi_j(\te+\om) = \tilde\eta_j(\te),
\]
where $\tilde\eta_j=\frac{1}{\al}\eta_j$, which is only solvable
if $\aver{\tilde\eta_j}=0$. In this case, we encounter the
small-divisors $1- e^{\bi 2\pi k\om}$. This complicates the meaning 
of the solution of \eqref{eq:coho}---and consequently of the
solution of \eqref{eq:cohoFT}. We overcome this difficulty using the standard
argument from \cite{Russmann75} of assuming 
the frequency $\om$ satisfies a Diophantine condition. Then,
assuming that the function $\eta$ is analytic in a domain,
\eqref{eq:cohosol} defines a function $\xi$ analytic in a
slightly smaller domain. In \cite{Russmann75},  there  are very
precise estimates on the size of $\xi$ in the smaller domain in
terms of the size of $\eta$ in the original domain.
In this paper, focused on algorithms, we will not discuss
estimates which are of course crucial to establish convergence
and will be presented in a future theoretical paper.

%The non-resonance conditions control the small divisors appearing in
%\eqref{eq:cohosol}. More specifically, when
%$\abs{\al}=\abs{\be\la^j}$ the cohomological equation
%\eqref{eq:coho} is a small-divisors equation which is only
%solvable if $\aver{\eta_j}=0$. In other words, we have an obstruction
%to the formal solution of \eqref{eq:coho}---and consequently of
%\eqref{eq:cohoFT}.  Furthermore, in the small-divisors case, for
%the series \eqref{eq:cohosol} to converge to a real-analytic
%function we need $\om$ to be Diophantine \cite{Russmann75}.
%

\section{Whiskered tori} \label{sec:invtori}
Let us consider $(d+1)$-dimensional invariant tori
$\hat{\cK}\subset U $, where $d=n-2$, with frequency vector
$\hat\om\in\RR^{d+1}$. For a suitable $T$, we can look for a
codimension$-1$ torus $\cK\subset\hat\cK$---invariant under
a time-$T$ map---that generates $\hat\cK$. See
e.g. \cite{HM21,FMHM24}. First, let us define
$\hat{\om}:=\tfrac{1}{T}(\om,1)$, and assume $\om\in\RR^d$ to
be Diophantine, see Definition \ref{def:diophantine}.

Then, we look for parameterizations $K:\TT^d\to U $
that conjugate the dynamics in $\cK=K(\TT^d)$ under $\phi_T$
to a rigid in $\TT^d$ with rotation vector $\om$. That is, we require the
parameterization to satisfy the following invariance equation
\begin{equation} \label{eq:invK}
    \phi_T\comp K - K\comp \Rom=0,
\end{equation}
where the map $\Rom$ defines the dynamics in $\TT^d$ and is
given by:
\begin{align*}
   \Rom:\quad \TT^{d} &\longrightarrow
   \TT^d \\
\te&\longmapsto \te+\om.
\end{align*}
An invariant torus $\cK$ is whiskered if it has directions that
contract exponentially fast in the future and in the past under
iteration of the differential of the map. These directions are the
invariant bundles of the torus and together with the tangent to 
the torus and its symplectic conjugate, span the tangent bundle 
of $ U $ restricted to $\cK$. Tangent to the exponentially contracting
directions, the torus has attached invariant manifolds that converge
exponentially fast to the torus---the whiskers.    

To make the definition of whiskered tori  precise, let us  
define the {\em transfer matrix}
$M:\TT^{d}\to\RR^{2n\times2n}$ defined on the graph of $\K$
as
\begin{equation*}\label{eq:defM}
M(\te) = \DzphiT\comp\K(\te).
\end{equation*}
In order to study the dynamics under iteration of $M$, for $n>0$ let us
define 
\begin{equation*}
M^{(n)} = 
M\comp\Rom^{\comp (n-1)}\dots M\comp\Rom M,
\end{equation*}
where
\[
\Rom^{\comp n}:=\underbrace{\Rom\comp\cdots\comp\Rom}_{n\rm\ times}.
\]
Observe that $M$ satisfies the cocycle relation
\begin{equation*}
        M^{(n+m)} = M^{(n)}\comp\Rom^{\comp m}M^{(m)},
\end{equation*}
which we can use to define $M^{(-n)}$ as 
\[
        M^{(-n)} =
        \left(M^{(n)}\comp\Rom^{\comp(-n)}\right)^{-1},
\]
where
\[
   \Rom^{\comp (-n)}:=\underbrace{R_{-\om}\comp\cdots\comp
   R_{-\om}}_{n\rm\ times}.
\]

\begin{Def}\label{def:whisktori}
An invariant torus $\cK=\K(\TT^{d})$ is said to be whiskered
if for every $\te\in \TT^{d}$, $T_{\K(\te)}U$ admits an
invariant splitting
\[
   T_{\K(\te)}U =E_{\K(\te)}^s\oplus E_{\K(\te)}^u\oplus
   E_{\K(\te)}^c  
\] 
such that there exist constants $C>0,
\mu_1<1$, and $\mu_2>1$, satisfying $\mu_1\mu_2<1$ and

\noindent\begin{minipage}{\textwidth}
   \begin{align*}
    &v \in E^s_{\K(\te)} \Longleftrightarrow\quad |M^{(n)}(v)|
      \leq C\mu_1^n |v| \quad &&n > 0,\\
    &v \in E^u_{\K(\te)} \Longleftrightarrow\quad |M^{(n)}(v)|
      \leq C\mu_1^{-n} |v| \quad &&n < 0,\\
    &v \in E^c_{\K(\te)} \Longleftrightarrow\quad |M^{(n)}(v)|
      \leq C\mu_2^{|n|} |v| \quad &&n \in \mathbb{Z}.
\end{align*}
\end{minipage}
\end{Def}
In our case, we can take $\mu_2$ arbitrarily close to $1$, taking $C$ 
sufficiently large.  As it is well known, the tangent vectors 
in the torus may grow polynomially with the number of iterations. 

\begin{remark}
The transfer matrix $M$ induces a {\em transfer operator} 
whose spectral properties are directly linked with the dynamical
properties of the invariant torus $\cK$. We will not follow such
approach here but it is useful to acknowledge this relationship.
For some references see e.g.,
\cite{Mather68,SackerS74,HirschPS70,Mane78,mamotreto}.
\end{remark}
\begin{remark}
When $\K$ parameterizes an invariant torus,
\[
   M^{(n)} = \Dif\phi_T^{\comp n}\comp\K,
\]
where
\[
\phi_T^{\comp
n}:=\underbrace{\phi_T\comp\cdots\comp\phi_T}_{n\rm\ times} =
\phi_{nT}.
\]
\end{remark}

\subsection{Invariance equation for whiskers}
We are interested in an invariant manifold $\hat\cW\subset U$, containing
the rotational torus $\hat\cK$ and its associated rank$-1$
whisker (that is  whiskers for which the dimension of the stable and 
unstable direction is $1$). As in 
the previous section, we consider a codimension$-1$ manifold
$\cW\subset\hat\cW$, invariant under the time$-T$ map, that
generates $\hat\cW$. In particular, we have that $\cK\subset\cW$.

We look for a parameterization
$W:\TT^{d}\times\RR\to U $ such that under the map $\phi_T$, the
parameterization conjugates the dynamics on the torus to a rigid rotation
and the dynamics on the whisker to a constant contraction. 
The value of the contraction is one of the unknowns of the procedure. 

A completely analogous argument works for the unstable direction 
using the inverse dynamics $\phi_{-T}$. Therefore, in what follows, we
assume $\abs{\la}<1$. The conjugacy condition translates into 
$W$ satisfying the invariance equation
\begin{equation} \label{eq:invW}
    \phi_T\comp W - W\comp \Rw =0,
\end{equation}
where the map
\begin{align*}
\Rw:\quad \TT^{d}\times\RR &\longrightarrow \TT^{d}\times\RR \\
(\te,s)&\longmapsto (\te+\om,\la s)
\end{align*}
represents the internal dynamics of $\cW=\W(\TT^d\times\RR)$
under $\phi_T$. Note that at $s=0$, we recover the invariance
equation for $\cK$ which is parameterized by $K(\te) :=
W(\te,0)$.

We emphasize that, in \eqref{eq:invW} the unknowns are both $W$ and $\la$. 
We consider $\om$ and $T$ fixed. The iterative step will involve 
adjusting both $\W$ and $\la$. 

If we differentiate the invariance equation
\eqref{eq:invW}, we obtain
\begin{equation} \label{eq:invDW}
\left(\Dif\phi_T \comp W\right)  \Dif W = \left(\Dif W\comp
   \Rw\right) \Dif\Rw,
\end{equation}
where
\[
\Dif\Rw=
\left(\begin{array}{c|c}
I_{d} & \\\hline
 & \la
\end{array}
\right).
\]
This expression reveals that the invariance equation for $\cW$
carries a bundle invariant under $\Dif\phi_T$: the tangent bundle
of $\cW$. In what follows,
for the sake of notational simplicity, we will avoid the use of
parenthesis in formulas as in \eqref{eq:invDW}, where composition
precedes products.

\begin{remark}\label{rem:uniqueness}
Solutions to \eqref{eq:invW} are not unique---we can
scale the coordinate $s$ by some scaling factor $\rho$
and obtain the reparameterization
$W^{\rho}(\te,s):=W(\te,\rho s)$. There is also a lack of
uniqueness in the parameterization of the torus. If $\W$ is a
solution of \eqref{eq:invW}, for any $\al\in\TT^d$ and $\tau\in\RR$,
$\W^{\al,\tau}(\te,s):=\phi_\tau\comp\W(\te+\al,s)$ is also a 
solution. 

Note that all these parameterizations give the same torus 
but there is freedom in the choice of scale and origin of angle in 
the reference manifold domain of the parameterization.
To obtain a Newton method with locally unique solutions, 
we will chose some normalization conditions that fix the origin 
of coordinates in $\theta$ and the scale in $s$. 
All these normalization conditions are equivalent for the mathematics,
but some of them lead to algorithms that are less sensitive to round-off error. 
We will specify the normalization choices in  \ref{sec:iter}.
\end{remark}

\begin{remark}
It is possible to find a parameterization $\W_1:\TT^{d}\to T_{\cK} U $ 
of the invariant bundles $\W_1$ by looking for a solution
of the following invariance equation 
\begin{equation}
  \Dif \phi_T \comp \K\, \W_1 -  \W_1\comp\Rom \la= 0,
\end{equation}
see \cite{HM21}. The bundle $\W_1$ can be considered as a first
order approximation of $\cW$.
\end{remark}

\begin{remark}
If the vector field $X$ is reversible, as in the CRTBP,
we can obtain parameterizations of the unstable whisker from that
of the stable whisker. More specifically, if there exists an
involution $\varrho:\RR^{2n}\to\RR^{2n}$ satisfying
\[
X\comp\varrho = -\Dif\varrho X,
\]
then we can obtain a parameterization of the unstable whisker, say $\tilde \W$, from
that of the stable whisker, and vice versa, as $\tilde\W=\varrho\comp\W$.
Note that the parameterization given by $\W(\te,0)$ is
different from the parameterization given by $\tilde W(\te,0)$.
Nonetheless, both $\W(\te,0)$ and $\tilde W(\te,0)$ give the 
same torus $\hat\cK$, see the beginning of Section \ref{sec:invtori}.
\end{remark}

\begin{remark} 
We can consider \eqref{eq:invW} also as including a normal form 
for the dynamics in the stable  whisker. 
In contrast to the Birkhoff normal form  developed in a neighborhood 
of the torus, the normal form on the stable whisker 
does not have any Birkhoff invariants. Therefore, the formalism used 
here can  be used rather generally. 
\end{remark} 

\begin{remark} 
Equation \eqref{eq:invW} involves only functions of 
the dimension of the stable manifold. It
nevertheless, together with \eqref{eq:invDW} and the adapted frame 
constructed in the next section, gives approximate control of 
what happens in a neighborhood of the stable manifold 
with a good accuracy. 
Using a full normal form around the torus produces higher
accuracy for the transitions but a full normal forms requires
using functions of a number of variables equal to the dimension
of the phase space.
\end{remark} 

\begin{remark} 
The formalism used in this paper generalizes to higher
dimensional (un)stable manifolds,
but some new phenomena appear. If the motion has real $\la_i$, 
there  could be resonances (the product of some $\la_i's$ is
another). Such resonances require only a small change  and is not a stable phenomenon in
the sense that, if they are present in a system, they could be 
destroyed by a small perturbation. However if some of the $\la_i$ are
a complex conjugate pair, there appear extra small divisors (the first Melnikov 
condition) and a new formalism is needed.  The difficulties are 
larger if one allows elliptic directions. 
Some algorithms for elliptic tori based on other methods have been developed in 
\cite{JorbaV97, LuqueV11, CaraccioloFH25}

We hope to come back to these problems and develop algorithms 
similar to this paper's  to develop 
simultaneous computations of tori and normal forms.

\end{remark}

\subsection{Geometric properties and adapted frames for $\cW$}
Recall that, according to \eqref{eq:invDW}, $\Dif W$ is invariant under
$\Dif\phi_T$. Additionally, a quick observation reveals that
\begin{equation}\label{eq:invX}
   \Dif \phi_T\comp \W\, X\comp \W = X\comp \W\comp \Rw,
\end{equation}
where we used the commutativity of flows and the invariance of $\cW$.
Consequently $X\comp W$ is also invariant under the differential of
$\phi_T$ on $\cW$.

Let us define a frame $\L:\TT^d\times\RR\to\RR^{2n\times n}$ for
the tangent bundle of $\hat\cW$ on $\cW$, which is a subframe of
$T_\cW U$, as follows
\begin{equation}\label{eq:L}
L := \Big(\Dif_\te W~\big |~ X\comp W ~\big |~ \Dif_s W \Big).
\end{equation}
\begin{proposition}\label{prop:Lag-L}
The subframe $\L$ in \eqref{eq:L} is a Lagrangian
subframe.
\end{proposition}
\begin{proof}
That $L$ is a Lagrangian subframe is equivalent to 
\begin{equation}
\label{eq:L_lag}
\Om_L:=L^{\ttop} \Om\comp W  L=0.
\end{equation}
First note that from \eqref{eq:invDW} and \eqref{eq:invX}
\begin{equation}\label{eq:invL}
\Dif\phi_T\comp W  L - L\comp\Rw  \La =0,
\end{equation}
where
\[
\La:=
\left(\begin{array}{c|c}
      I_{n-1} & \\\hline
              &\la
   \end{array}
   \right).
\]
Then, using the symplecticity of $\phi_T$, \eqref{eq:invW}, and
\eqref{eq:invL} we obtain
\begin{equation}\label{eq:cohoOmL}
   \OmL - \La^\ttop\OmL\comp\Rw  \La = 0.
\end{equation}
Let us split $\Om_L$ into blocks of sizes $(n-1)\times(n-1),
(n-1)\times1, 1\times(n-1),$ and $1\times1$ as 
\begin{equation*}
   \Om_L = \begin{pmatrix}\Om_L^{11} & \Om_L^{12}\\ \Om_L^{21} &\Om_L^{22}
   \end{pmatrix}.
\end{equation*}
Then, \eqref{eq:cohoOmL} translates into
\begin{align*}
   \Om_L^{11} -\phantom{\la} \Om_L^{11}\comp\Rw\phantom{\la} &= 0,\\
   \Om_L^{12} -\phantom{\la} \Om_L^{12}\comp\Rw\la &= 0,\\
   \Om_L^{21} - \la\Om_L^{21}\comp\Rw\phantom{\la} &= 0,\\
   \Om_L^{22} - \la\Om_L^{22}\comp\Rw\la &= 0,
\end{align*}
and if we use Fourier-Taylor series, we obtain 
\begin{align*}
   \sum_{j,k}\left(\Om_L^{11}\right)_{jk}\left(1-e^{\bi2\pi
   k\om}\la^j\right)e^{\bi2\pi k \te}s^j&=0,\\   
   \sum_{j,k}\left(\Om_L^{12}\right)_{jk}\left(1-e^{\bi2\pi
   k\om}\la^{j+1}\right)e^{\bi2\pi k \te}s^j&=0,\\
   \sum_{j,k}\left(\Om_L^{21}\right)_{jk}\left(1-e^{\bi2\pi
k\om}\la^{j+1}\right)e^{\bi2\pi k \te}s^j&=0,\\
   \sum_{j,k}\left(\Om_L^{22}\right)_{jk}\left(1-e^{\bi2\pi
   k\om}\la^{j+2}\right)e^{\bi2\pi k \te}s^j&=0
\end{align*}
for all $\te$ and $s$. Hence,
$\Om_L^{11}=\left(\Om_L^{11}\right)_{00}$, $\Om_L^
{12}=\Om_L^{21}=\Om_L^{22}=0$. Hence, $\Om_L$ is constant.
We can then inspect the Lagrangian character of $L$ on the torus $\cK$. That is,
\begin{equation*}
\Om_L(\te,s) = \Om_L(\te,0).
\end{equation*}
Furthermore, it is known, see e.g. \cite{HM21}, that $L$
generates a Lagrangian subspace on $T_\cK U $,
i.e., $\Om_L(\te,0)=0$. Therefore, $\Om_L=0$. It remains to be shown
that $\Dif W$ and $X\comp W$ generate an $n-$dimensional subspace. This
can be done without much work by looking at the vector field
invariance equation for the whisker generated by $\cW$. 
\end{proof}
\begin{remark}
Note that the fact that $L$ is a Lagrangian subframe implies that
$\cW$ is an isotropic manifold, i.e, 
\[
   (\Dif \W)^\ttop\Om\comp\W\Dif\W = 0,
\]
which is a well known result \cite{Zehnder76}.
\end{remark}
Let us now complement the subframe $L$. In
particular, we can complement $L$ with a normal bundle
$N:\TT^{d}\times\RR\to\RR^{2n\times n}$ to a full symplectic basis of
$T_\cW  U $ given by a frame
$P:\TT^{d}\times\RR\to\RR^{2n\times 2n}$ as
\begin{equation} \label{eq:P}
P:=\left(L\  \ N \right),
\end{equation}
where
\begin{equation}\label{eq:N}
   \N:= \J\comp\W  \L\hspace{1pt} G_L^{-1}
\end{equation}
and
\[
 \GL:=L^\ttop \G\comp\W\ L.
\]

The formula \eqref{eq:N} gives an explicit Lagrangian subspace
complementing the Lagrangian subspace of the tangent.  This is an
explicit  implementation for our case of the well known result
that every Lagrangian subspace can be complemented with another
Lagrangian subspace to give a fully symplectic space
\cite{Cannas01}. 

The frame $P$ has geometric
properties that are fundamental in our methods. In particular, by
construction, $P$ is symplectic with respect to the standard
symplectic form. That is,  
\begin{equation}\label{eq:symp}
   \P^\ttop \Om\comp \W \P = \OmO = 
   \left(\begin{array}{c|c}
         & -I_n\\\hline
      I_n & \end{array}\right).
\end{equation}
This identity shows that, on $\cW$, the symplectic form in
the coordinates given by $\P$ reduces to the standard symplectic
form. If we right-multiply \eqref{eq:symp} by $\P^{-1}$ and
then left-multiply it by $-\OmO$, we obtain the following
formula for the inverse of $\P$
\begin{equation}\label{eq:invP}
   \P^{-1}=-\OmO\PT\Om\comp\W.
\end{equation}
Also, the differential of $\phi_T$ on $\cW$ in the coordinates given by
$P$ reduce to an upper triangular matrix constant along the diagonal.
That is
\begin{equation}\label{eq:AR}
   (\P\comp\Rw)^{-1} \Dif\phi_T\comp W P =
   \left(\begin{array}{c|c}
         \La & S \\\hline 
             & \mybox{\La^{-\ttop}}
      \end{array}
   \right)
\end{equation}
where $S:\TT^{d}\times \RR\to\RR^{n\times n}$ is referred as the {\em torsion} 
or shear of $\cW$ and is defined as
\begin{equation}\label{eq:Sdef}
S = \N^\ttop\comp\Rw\ \Om\comp \W\comp\Rw\ \Dif\phi_T\comp\W \N.
\end{equation}
The reducibility of $\Dif \phi_T$ on $ \cW$, as expressed in
\eqref{eq:AR}, follows from
the Lagrangianity of $\L$, the invariance of $\L$, and the
symplecticity of $\P$, i.e., from 
\eqref{eq:L_lag}, \eqref{eq:invL}, and \eqref{eq:symp}.
This property is known
as {\em automatic reducibility} and it holds in a variety of cases,
see e.g. \cite{mamotreto,HuguetLS12}.

\begin{remark}
\label{rem:normalize} 
When computing tori and bundles, see \cite{HM21,FMHM24}, it is
possible to find changes of variables that reduce the torsion $S$
to a constant matrix on the torus $\cK$. This might improve 
numerical stability. Also, with some extra work, one could 
obtain both the stable and the unstable bundles simultaneously as
part of the frame $\P$. 
The reduction of the torsion to constant is obtained by applying
a linear symplectic transformation to the frame $\P$. In the case
of whiskers,
it is possible to reduce the torsion by means of a linear
symplectic transformation to
\begin{equation*}
\S = \begin{pmatrix}  \langle S^1_0 \rangle &\langle
   S^2_1\rangle s\\ \langle S^3_1 \rangle s & \langle S^4_2\rangle
s^2\end{pmatrix},
\end{equation*}
where the splitting of $\cS$ is in blocks of
size $(n-1)\times (n-1), (n-1)\times 1, 1\times (n-1),$ and $1\times
1$. Therefore, the torsion $S$ can be reduced to constant in
$\TT^{d}$ and only with terms up to order 2 in
$s$. Note that this involves solving a cohomology equation 
but then, the multiplication by the torsion becomes multiplication
by a constant rather than multiplication by a function. 
\end{remark}

\section{Computation of whiskers}
\label{sec:Newton}
For the computation of $\cW$, assume we have a parameterization
$\W:\TT^{d}\times\RR\to U $ and $\la$ that satisfy \eqref{eq:invW}
approximately. Let us define the error in the invariance of $\cW$
as $E:\TT^d\times\RR\to\RR^{2n}$ given by 
\begin{equation}\label{eq:error}
E:=\phiT\comp \W  - \W\comp \Rw.
\end{equation}
Our objective is to find a map $\De
W:\TT^{d}\times\RR\to\RR^{2n}$ and $\De\la\in\RR$, such that for
$\bar \W=\W+\DeW$ and $\barla=\la+\Dela$, the new error $\bar E$
satisfies $\|\bar E\|' \sim\cO(\|E\|^2)$ for some
suitable norms $\|\ \|$ and $\| \ \|'$\footnote{As it is well known in KAM 
theory, the norm $\| \ \|'$ is a norm corresponding to 
analytic functions in a smaller domain. The $\cO$ 
term also involves explicit quantitative bounds in terms of 
the domain loss.}
.
More specifically, we will construct sequences for $\W$ and $\la$
by means of a quadratically convergent iterative scheme such that
in the limit we have an exact solution of \eqref{eq:invW}.

In the approximately invariant case, where the invariance equation is
satisfied approximately, the range of $\W$ will be approximately
isotropic. Consequently, the frame $\P$ will be approximately symplectic
and the reducibility of the dynamics will hold approximately. The
error in these properties is actually controlled by the error in the
invariance equation \eqref{eq:error}. 
In the iterative scheme, we can neglect these reducibility errors
as their contribution to the new error is of second order in $E$
in the sense of Nash-Moser.

\subsection{Iterative scheme}\label{sec:iter}
For the iterative procedure, we require that in the linearized invariance
equation, the corrections $\DeW$ and $\Dela$ are such that they cancel
$E$. That is, we require
\begin{equation*} 
   \Dif\phi_T\comp \W  \DeW - \DeW\comp\Rw = -E +\pd_s
   \W\comp\Rw \Dela s.
\end{equation*}
If we express $\De W$ in the coordinates given by $P$, i.e.,  
$\DeW = \P\xi$, and use \eqref{eq:AR}, we obtain
\begin{equation*}
   \left(\P\comp\Rw\right)^{-1}\Dif\phi_T\comp \W \P\xi - \xi\comp\Rw =
\left(\P\comp\Rw\right)^{-1}\bigg(-E + \pd_s \W\comp\Rw \De\la s\bigg).
\end{equation*}
We can now use \eqref{eq:invP} and, after neglecting
quadratically small errors, we obtain
\begin{equation}\label{eq:sys1} 
\left(
\begin{array}{c|c}
   \La & S\\\hline
       & \mybox{\La^{-\ttop}}
\end{array}   
\right) \xi - \xi\comp\Rw
=\eta + e \De\la s,
\end{equation}
where
\begin{equation}\label{eq:defeta}
\eta := \Omega_0 \left(P\comp\Rw\right)^\ttop\Omega\comp W\comp\Rw
~ E, \quad
e:=\begin{pmatrix} e_n\\ 0_n \end{pmatrix},
\end{equation}
$e_n=(0,...,1)^\ttop\in\RR^n$, and $0_n=(0,...,0)^\ttop\in\RR^n$.
If we split $\xi$ and $\eta$ into $(n-1)\times 1
\times(n-1)\times 1$ components as
\begin{equation*}
            \xi=\begin{pmatrix}
            \xi^1\\
            \xi^2\\
            \xi^3\\
            \xi^4
            \end{pmatrix} , \quad
            \eta=\begin{pmatrix}
            \eta^1\\
            \eta^2\\
            \eta^3\\
            \eta^4
            \end{pmatrix}, \\           
\end{equation*}
and $S$ into blocks of sizes $(n-1)\times(n-1), (n-1)\times 1,
1\times(n-1)$, and $1\times 1$ as
\begin{equation*}
S=\begin{pmatrix}
S^1 & S^2\\
S^3 & S^4
\end{pmatrix},
\end{equation*}
we obtain the following system of equations
\begin{align}
   \xi^1+S^1\xi^3 +S^2\xi^4-\xi^1\comp\Rw &= \eta^1,\label{eq:coho1} \\
   \la\xi^2+S^3\xi^3+S^4\xi^4-\xi^2\comp\Rw &= \eta^2+\De\la s, \label{eq:coho2}\\
   \phantom{ \la\xi^2+S^3\xi^3+}\xi^3-\xi^3\comp\Rw &= \eta^3,\label{eq:coho3}\\
   \la^{-1}\xi^4-\xi^4\comp\Rw &= \eta^4.\label{eq:coho4}
\end{align}
Let us expand $\xi, S,$ and $\eta$ in Fourier-Taylor series, see Section
\ref{sec:FT}. 
We can first solve for $\xi^4$ as described in Section
\ref{sec:coho}.  
If $\langle \eta^3_0 \rangle=0$, we can solve for $ \xi^3$ up to
an arbitrary additive constant $\langle\xi^3_0\rangle$ that will
be chosen later. Using the exactness of $\phi_T$, it is shown in \cite{HM21}
that $\aver{\eta^3_0}$, is  quadratically small in $E$, 
so we can just set it to zero in the iterative 
step without affecting the quadratic convergence. 
Once we have solved for $\xi^4$ and for $\xi^3$ 
with $\langle \xi^3_0 \rangle$ free, we obtain the following
equations for $\xi^1$ and $\xi^2$ 
\begin{align}
\phantom{\la}\xi^1 - \xi^1\comp\Rw &= \de^1, \label{eq:xi11}\\
\la\xi^2 - \xi^2\comp\Rw &= \de^2, \label{eq:xi22}
\end{align}
where 
\begin{align}
\de^1&:= \eta^1-S^2\xi^4 - S^1\xi^3,\label{eq:de1}\\
\de^2&:=\eta^2 - S^4\xi^4 - S^3\xi^3 + \De\la s.\label{eq:de2}
\end{align}
Note that in \eqref{eq:xi11} we have an obstruction at order $0$ in $s$ 
and in \eqref{eq:xi22} we have an obstruction at order $1$ in $s$, see
Section \ref{sec:coho}.
Therefore, in order to solve for $\xi^1$ 
and for $\xi^2$, we need that $\langle\de^1_0 \rangle=0$ and
$\langle \de^2_1\rangle=0$. We can
use the freedom in $\langle \xi^3_0\rangle$ and the increment of  parameter
$\De \la$ to adjust these averages. 
Let 
\begin{equation}\label{eq:xi3}
\xi^3 = \tilde \xi^3 + \langle \xi^3_0 \rangle,
\end{equation}
where $\langle \tilde \xi^3_0 \rangle=0$, and let us define 
\begin{align}
\zeta^1 &:= \eta^1 -S^2\xi^4 -
S^1\tilde\xi^3,\label{eq:zeta1}\\
\zeta^2&:= \eta^2 - S^4\xi^4 -
S^3\tilde\xi^3.\label{eq:zeta2}
\end{align}
We then choose $\langle \xi^3_0\rangle$ and $\De\la$ such that
\begin{equation}\label{eq:linsys}
\left(
\begin{array}{c|c}   
\langle S^1_0\rangle &  \\\hline
\mybox{\langle S^3_1}\rangle & -1
\end{array}
\right)
\begin{pmatrix}
\langle \xi^3_0\rangle \\ \De\la
\end{pmatrix} = \begin{pmatrix}
\langle \zeta^1_0\rangle \\
\langle \zeta^2_1\rangle
\end{pmatrix},
\end{equation}
given the non-degeneracy twist condition
\begin{equation}\label{eq:nondeg}
\rm{det}
\langle S^1_0\rangle 
\neq 0.
\end{equation}
Recall that \eqref{eq:xi11} at order 0 and \eqref{eq:xi22} at
order 1 are small divisors cohomological equations---therefore,
$\langle \xi^1_0 \rangle$ and $\langle\xi^2_1\rangle$ are free.
This underdeterminacy reflects the freedom in the choice of $K$,
and in the choice of the first order approximation of $\cW$, i.e.,
the freedom in the scaling of the invariant bundle of $\cK$, see
Remark \ref{rem:uniqueness}. A simple
choice is to take $\langle \xi^1_0 \rangle=0$ and
$\langle\xi^2_1\rangle=0$.  
Then, we can solve for $\xi^1$ and $\xi^2$ as described in
Section \ref{sec:coho}. Lastly, 
we obtain the corrected parameterization
and the corrected rate of contraction as $\bar \W= W + P \xi$ and
$\barla=\la + \Dela$, respectively.

\begin{remark}\label{re:scaling}
As already mentioned, there is certain freedom in the
parameterization of $W$---we can
scale the coordinate $s$ by some scaling factor $\rho$
and obtain the reparameterization
$W^{\rho}(\te,s)=W(\te,\rho s)$. Using the
Fourier-Taylor representation of Section \ref{sec:FT},
the coefficients of $\W^\rho$ are
${\W^\rho}_j(\te)=\rho^j\W_j(\te)$.

All such  parameterizations are mathematically equivalent but the
scaling can have effects on the numerics. If one chooses scaling factors 
such that coefficients grow or decrease to zero too fast with the order,
the algorithm might be affected by round-off error. 
In order to avoid this, we can
take $\rho=1/r$, where $r$ is an estimate of the radius of convergence
of the series in $s$ that we can obtain with a root test.
\end{remark}

\begin{remark}\label{rem:correct-terms}
Suppose we represent $\W$ as a Fourier-Taylor
series, see Section \ref{sec:FT}, and assume we have,
for instance, the torus and the bundle. That
is, we have the terms $\W_0$ and $\W_1$.
Since the iterative procedure is quadratically
convergent, after $it$ iterations we have
$2^{it+1}-1$ correct terms for $\W$. 
\end{remark}
We conclude this section with algorithm  \ref{algo:newton} for
one step of the iterative scheme.
\begin{algorithm}
\caption{Iterative step}\label{algo:newton}
\raggedright
Given $\W$ and $\la$ satisfying \eqref{eq:invW} approximately, compute
the corrected parameterization $\bar\W$ and rate of contraction
$\bar\la$ by
following these steps:
%\vspace{-1em}
\begin{algorithmic}[1]
\STATE Construct $L,N,$ and $P$ from \eqref{eq:L},
\eqref{eq:N}, and \eqref{eq:P}, respectively.   
\STATE Compute $S$ from \eqref{eq:Sdef}. (Optionally: Normalize it as detailed in Remark~\ref{rem:normalize})
\STATE Compute the error in the invariance equation, $E$, from \eqref{eq:error}. \STATE Construct $\eta$ from the definition \eqref{eq:defeta}.  
\STATE Solve \eqref{eq:coho4} to obtain $\xi^4$.
\STATE Solve \eqref{eq:coho3} to obtain
$\tilde\xi^3$ with $\langle\tilde\xi^3_0\rangle=0$.
\STATE Compute $\zeta^1$ and $\zeta^2$ from \eqref{eq:zeta1} and
\eqref{eq:zeta2}, respectively. 
\STATE Solve the linear system \eqref{eq:linsys} given the non-degeneracy
condition \eqref{eq:nondeg} in order to obtain $\langle \xi^3_0\rangle$
and $\De\la$. 
\STATE Construct $\xi^3$ from \eqref{eq:xi3} and compute $\de^1$ and
$\de^2$ from \eqref{eq:de1} and \eqref{eq:de2}, respectively. 
\STATE  Set $\langle \xi^1_0\rangle=0$ and $\langle\xi^2_1\rangle=0$.
Solve \eqref{eq:xi11} and \eqref{eq:xi22}
to obtain $\xi^1$ and $\xi^2$.
\STATE Set $\bar \W\leftarrow \W + P\xi$ and $\barla\leftarrow\la+\De\la.$
\end{algorithmic}
\end{algorithm}
Note that all steps of the algorithm are vector operations 
either in Fourier representation or in grid representation.
Therefore, the operation count is linear in the size of data and no
large matrices (of dimension the size of the data) need to be stored and 
much less inverted.  
Note also that all the operations are well structured. They can 
be implemented  in a line in a modern vector language or by developing 
a library manipulating Fourier-Taylor and grid-Taylor representations of 
functions. 

\subsection{Continuation with respect to $T$}\label{sec:cont}
A Newton method gives rise to a continuation strategy. 
The most elementary is taking a well computed zero for 
a value of the parameter as an initial guess of 
the algorithm to compute a zero for an slightly different value,
but one can also use extrapolations, etc. Note that the a-posteriori 
theorem guarantees that these continuation methods will progress 
until parameter values where the hypothesis of the a-posteriori theorem 
do not hold, which indicates breakdown of the tori.

In this section, we discuss a particularly efficient continuation 
method on the parameter $T$, the time used for the flow map. 
This is equivalent to multiplying the frequency of the torus 
which is the usual KAM method of exploring what happens in energy 
surfaces. The continuation in the parameter $T$ is significantly 
better behaved that continuation on other parameters since we
can obtain explicit extrapolation formulas and reuse many of 
the calculations. 

Assume we have a $\phi_T$ invariant torus and its whisker. That
is, we have $(W,\la)$ satisfying \eqref{eq:invW}.
Observe that \eqref{eq:invW} defines implicitly $(\W,\la)$ as
functions of $T$. 
We are now
interested in another solution $(W^*,\la^*)$ of \eqref{eq:invW},
defining a $\phi_{T^*}$ invariant whisker $\cW^*$ for
some $T^*$ close to $T$. As is standard in continuation methods, we
provide in this section a method to compute the derivatives of
$(W,\la)$ with respect to $T$ from where we can compute a first
order approximation of $(W^*,\la^*)$. 

Let Eq.~\eqref{eq:invW} define implicitly $W$ and $\la$ as
functions of $T$.  Then, differentiating \eqref{eq:invW} with respect to
$T$, we obtain
\begin{equation*}
\Dif_z\phi_T\comp \W  \pd_T W +X\comp W\comp\Rw - \pd_T
\W\comp\Rw -\Dif_s W\comp\Rw  \pd_T \la s =0.
\end{equation*}
If we express $\pd_T W$ in the coordinates given by the frame as
$\pd_T W=P\xi$, and left-multiply the
previous expression by $\left(P\comp\Rw\right)^{-1}$, we obtain
the following system of cohomological equations 
\begin{equation} \label{eq:cont-sys}
\left(
\begin{array}{c|c}
\Lambda & S\\\hline
        & \mybox{\Lambda^{-\ttop}}
\end{array}
\right)
\xi - \xi\comp\Rw =\begin{pmatrix}
            0_{n-2}\\
            -1\\
            \pd_T \la\hspace{0.1em} s\\
            0_{n}
            \end{pmatrix}.
\end{equation}
This system is completely analogous to \eqref{eq:sys1}, but with 
\begin{equation}\label{eq:eta_cont}
\eta = \begin{pmatrix}
            0_{n-2}\\
            -1\\
            0_{n+1}
            \end{pmatrix}
\end{equation}
and $\pd_T\la$ instead of $\De\la$. In this case,
$\xi^4=0$ and $\xi^3=\langle\xi^3_0\rangle$. We can choose 
$\langle\xi^3_0\rangle$ and $\pd_T\la$ to adjust
averages for the equations for $\xi^1$ and $\xi^2$ to be 
solvable. That is, we choose $\langle\xi^3_0\rangle$ and
$\pd_T\la$ such that  
\begin{equation}\label{eq:linsys2}
\left(
\begin{array}{c|c}
\langle S^1_0\rangle &  \\\hline
\mybox{\langle S^3_1\rangle} & -1
\end{array}
\right)
\begin{pmatrix}
\langle \xi^3_0\rangle \\ \pd_T\la
\end{pmatrix} = \begin{pmatrix}
	-e_{n-1} \\
	0
\end{pmatrix},
\end{equation}
where $e_{n-1}$ is as in \eqref{eq:defeta}. 
Then, we solve for $\xi^1$ and $\xi^2$ as in the previous 
section.

The computation of derivatives of $(\W,\la)$ with respect to $T$
is summarized with algorithm \ref{algo:continuation}.
\begin{algorithm}[H]
\caption{Computation of derivatives with respect to
$T$}\label{algo:continuation}
\raggedright
Given $T$ and $(W,\la)$ satisfying \eqref{eq:invW},
find $\pd_T W$ and $\pd_T\la$ by following
these steps:   
\begin{algorithmic}[1]
\STATE Construct $L,N,$ and $P$ from \eqref{eq:L},
      \eqref{eq:N}, and \eqref{eq:P}, respectively.   
   \STATE Compute $S$ from \eqref{eq:Sdef}.
   \STATE Construct $\eta$ from the definition \eqref{eq:eta_cont}. 
   \STATE Set $\xi^4=0$. 
   \STATE Solve the linear system \eqref{eq:linsys2}  given the
   non-degeneracy condition \eqref{eq:nondeg} in order to
   obtain $\langle \xi^3_0\rangle$ and $\pd_T\la$.
   \STATE Set $\xi^3 = \langle \xi^3_0\rangle$.
   \STATE Compute $\de^1$ and $\de^2$ from \eqref{eq:de1} and
      \eqref{eq:de2}, respectively, with $\pd_T\la$ instead
      of $\De \la$. 
   \STATE Set $\langle \xi^1_0\rangle=0$, 
      $\langle\xi^2_1\rangle=0$, and solve
      \eqref{eq:xi11} and \eqref{eq:xi22} 
      to obtain $\xi^1$ and $\xi^2$, respectively.
   \STATE Set $\pd_T W=P\xi$.
\end{algorithmic}
\end{algorithm}

\subsection{Fundamental domain}\label{sec:domain}
Assume we have a parameterization $W$ given by a truncated
Fourier-Taylor series at order $N_T$. The series for $W$ will only be
valid in $\TT^{d}\times\mathcal{D}$ for some open interval
$\mathcal{D}\subset\RR$---we refer to $\cD$ as the fundamental
domain of $\W$.
Let us consider the following norm for the error \eqref{eq:error} 
\begin{equation}\label{eq:normE}
   \|E\|_{\cA}:= \sup_{\subalign{\theta&\in\TT^{d} \\
                s&\in\cA}} \|E(\theta,s)\|_{\infty},
\end{equation}
where $\cA\subset\RR$.
Then, we can define $\cD$ as the largest interval where  
\begin{equation}\label{eq:tol}
   \|E\|_{\cD}<\epsilon,
\end{equation}
holds for some threshold $\epsilon>0$ small. If we proceed formally,
we can construct a Fourier-Taylor series for $E$ of the form 
\begin{equation}\label{eq:serieE}
    E(\theta,s)=\sum_{j=0}^\infty E_j(\theta)s^j.
\end{equation}
Then, using the definition \eqref{eq:normE}, we have the inequality
\[
   \|E\|_{\cD}=\sup_{\subalign{\theta&\in\TT^{d} \\
    s&\in\cD}}\| \sum_{j=0}^\infty E_j(\theta) s^j\|_\infty
    \leq \sum_{j=0}^\infty\left( \sup _{\subalign{\theta&\in\TT^{d} \\
    s&\in\cD}} \| E_j(\theta)s^j\|_\infty\right)
\]
and if we define 
\begin{equation}\label{eq:normsE}
    \|E_j\|:=\sup_{\theta\in\TT^{d}
    } \|E_j(\theta)\|_\infty, \quad D:=\sup_{s\in\cD}|s|,
\end{equation}
we can estimate the fundamental domain by requiring
\begin{equation}\label{eq:normEin}
   \|E\|_{\cD}\leq \sum_{j=0}^{N_T}\|E_j\| D^j +
                \sum_{j=N_T+1}^\infty\|E_j\|D^j<\ep.
\end{equation}
Observe that if we want \eqref{eq:normEin} to hold,
necessarily for any $k$
\begin{equation}\label{eq:Ek}
\norm{E_k}D^k<\ep.
\end{equation}
Let us now inspect which term contributes the most in the sums of
\eqref{eq:normEin} so we can obtain a useful bound for $D$ from
\eqref{eq:Ek}.

Observe that $E_j$, for $j\leq N_T$, should be small since these are
the terms that we control with our iterative procedure. The terms
for $j>N_T$ are terms of order $j$ of $\phi_T\comp\W$ since we do
not have terms of order higher than $N_T$ for $\W$. Consequently,
we assume that the term that contributes the most in
\eqref{eq:normEin} is in the infinite sum.
Note that the radius of convergence for
the series \eqref{eq:serieE}, say $\tilde D$, satisfies
$D\leq\tilde D$.  Furthermore, $E$ is analytic so we can bound
each $E_j$ using Cauchy estimates by
\[
\norm{E_j}\leq\frac{C}{\tilde D^j},
\]
for some constant $C$. Consequently, in the infinite sum, we can
bound each term by
\[
   C\left(\frac{D}{\tilde D}\right)^j,
\]
from where we see that the largest upper bound corresponds to the
term $j=N_T+1$.
We can then obtain the following necessary condition for $D$ 
\begin{equation}\label{eq:estD}
   D<\left(\frac{\epsilon}{\|E_{N_T+1}\|}\right)^{\frac{1}{N_T+1}},
\end{equation}
from where we estimate the fundamental domain as $\cD = [-D,D]$.

Alternatively, for some grid of values $\{\te_l\}_{l\in{\rm I}}$,
where ${\rm I}$ is some index set, and some $\de > 0$ small
we can estimate $\|E\|_{\cA}$ numerically as
\begin{equation}\label{eq:discnormE}
   \|E\|_{\cA}\approx \max_{\subalign{l&\in {\rm I}\\
   s\in\ZZ&\de\cap\cA}}
   \left(\|\phi_T\big(W(\theta_l,s)\big) - W(\theta_l +\omega, \lambda
   s)\|_{\infty}\right).
\end{equation}
To estimate the fundamental domain, let us define
\[
   k = \max\{l\in\NN : \norm{E}_{[-l\de, l\de]}<\ep\},
\]
where the norm is approximated with \eqref{eq:discnormE}. Then,
we simply obtain $\cD$ as $\cD=[-k\de, k\de]$.

%To estimate the fundamental domain, construct a sequence
%$\{D_i\}_{i=0}$ where $D_{i+1} = D_i + \De D$, for  some $\De D$
%small. Then, we obtain the estimate $\cD=[-D_k,D_k]$ where
%$D_k\in\{D_i\}_{i=0}$ is the largest element in the sequence
%satisfying \eqref{eq:tol} for $\norm{E}_{\cD}$ estimated as in as
%in \eqref{eq:discnormE}. 

\section{Computational considerations}\label{sec:considerations}
In this section, we provide some insights into the
computational implementation of the methods discussed in Section
\ref{sec:Newton}.

\subsection{Function representations}
The algorithms of the previous section for the computation of
Newton and continuation steps are written in terms of scalar,
vector, and matrix valued functions defined in $\TT^d\times\RR$.
In view of the expressions for the solutions of cohomological
equations (Section~\ref{sec:coho}), a natural representation for
these functions is as Fourier-Taylor series as given in
\eqref{eq:FT}-\eqref{eq:Fseries}. In a computer, truncated
versions of such series can be stored.

For functions $f:\TT^d\to\RR$ (resp.~$f:\TT^d\rightarrow\RR^\ell$,
$f:\TT^d\rightarrow\RR^{\ell\times m}$) and some finite multi-index set
$\cI$ with elements in $\ZZ^d$, we denote {\em Fourier
representations} of $f$ by
\[
f\sim\sum_{k\in\cI}\hat
   f_{k}e^{\bi2\pi k \te},
\]
which is given by the finite set of coefficients $\{\hat
f_k\}_{k\in\cI}$. 

As it is usual in the numerical implementation of the
parameterization method in KAM-like contexts (see
\cite{mamotreto} for an overview), it is convenient to work
also with a grid representation of the function $f$.  For some
finite multi-index set ${\rm I}$ with elements in $\NN^{d}$, we also
consider representations of $f$ of the following form
\[
   f\sim\{f_{l}\}_{l\in{\rm I}},
\]
where $f_{l}=f(\te_l)$ and $\{\te_l\}_{l\in {\rm I}}$ is an
equally spaced grid of values in $\TT^{d}$ indexed by ${\rm
I}$. Let $N_F$ be the
cardinality of the set $\{f_{l}\}_{l\in{\rm I}}$, i.e., the number of
elements of the grid of $\TT^{d}$. Then, the 
$\{\hat f_{k}\}_{k\in \cI}$ coefficients can be obtained
from the set $\{f_{l}\}_{l\in {\rm I}}$ through the Discrete Fourier
Transform (DFT) and vice-versa through the inverse DFT---at a
computational cost of $O(N_F\log N_F)$ operations. We refer to the
set of values $\{f_{l}\}_{l\in{\rm I}}$ as a {\em grid
representation} of the function $f$. 

\begin{remark}
We remark  that the Fourier coefficients used 
in the  description and the DFT coefficients used 
in the implementation are not the same. 
Nonetheless, there is a close routine connection, see
\cite{HM21,Henrici79}, and this is just a minor bookkeeping detail. 
\end{remark}

For functions $f:\TT^{d}\times\RR\to\RR$
(resp.~$f:\TT^d\times\RR\rightarrow\RR^\ell$,
$f:\TT^d\times\RR\rightarrow\RR^{\ell\times m}$) 
and some positive natural order
$N_T$, we consider {\em Fourier-Taylor} representations of $f$ of the
following form
\begin{equation}\label{eq:FT representation}
   f\sim \sum_{j=0}^{N_T} f_j(\te)s^j \sim 
   \sum_{j=0}^{N_T}\sum_{k\in\cI}\hat f_{jk}e^{\bi2\pi k\te}s^j.
\end{equation}
Numerically, functions are represented by the finite set of coefficients
$\{\hat f_{jk}\}\subalign{&0\leq j\leq N_T\\ &k\in\cI}$.
It will be convenient to also work with a grid representation of
the $f_j$ coefficients in \eqref{eq:FT representation}.
We will also consider representations of $f$ of the following form
\[
   f\sim \{f_{jl}\}\subalign{&0\leq j\leq N_T\\&l\in{\rm I}},
\]
where $f_{jl}=f_j(\te_l)$ and $\{\te_l\}_{l\in{\rm I}}$ 
is an equally spaced grid of values in $\TT^d$.
We refer to the set of coefficients $\{f_{jl}\}\subalign{&0\leq
j\leq N_T\\&l\in{\rm I}}$ as a {\em grid-Taylor} representation of $f$.
Such representation is very well suited for the computation of
product of functions with variables in $\TT^d\times\RR$ and for
the (numerical) application of the flow and its differential to
such functions. That is, for the numerical computation of
$\phi_T\comp f$ and $(\Dif\phi_T\comp f)g$ for fixed $T$ and $f, g :
\TT^d \times \RR \to \RR^{2n}$.
The computation of the flow applied to Taylor polynomials 
is often called Jet Transport.

\subsection{Jet Transport}\label{sec:JT}
In the iterative scheme of Section \ref{sec:iter}, it is
necessary to compute Fourier-Taylor series for the flow and the
variational equations applied to functions defined in
$\TT^d\times\RR$. That is, we need to compute the
coefficients $\phi_j,\Phi_j:\TT^d\to \RR^{2n}$ of
the following Fourier-Taylor representations
\begin{gather}
   \phiT\comp\W\sim\sum_{j=0}^{N_T} \phi_j(\te)
  s^j,\label{eq:JT1m}\\
  \left(\Dif\phi_T\comp\W\right) u \sim
  \sum_{j=0}^{N_T}\Phi_j(\te) s^j,\label{eq:JT2m}
\end{gather}
for some arbitrary $u:\TT^{d}\times\RR\to\RR^{2n}$. 

We note that propagating polynomials 
up to high order is the same as computing the derivatives of
an evolution. This  is a standard numerical problem 
called {\em Jet Transport}---which has attracted interest 
in celestial mechanics since it enters in several other aspects such 
as computations of normal forms, bifurcations, etc. Several 
systematic approaches and surveys can be found in
\cite{jetESA, tesisPerez, PerezPMG15,  GimenoJT}.

The Jet Transport algorithm we use is based on taking a 
standard integrator of ODE's and overload the arithmetic
operations involved in  evaluating a step with operations on polynomials.
We think of the operations as the representative of derivatives.  
For the applications to the invariance equation, we use the grid-Taylor 
representation and compute the flow map of the jet on each point.
For the CRTBP, all the operations involved are 
arithmetic, but it is well known that there are fast implementations of
many other functions $(\sin,~ \cos, ~\exp,~\log,$ elliptic
functions, etc). Some comments on our actual implementation of
the Jet Transport technique can be found in Section
\ref{sec:comments-whiskers}.

% Let us define the Taylor series
Let us elaborate on our approach. Assume we have grid-Taylor representations of
$\W$ and $u$, i.e, we have $\{W_{jl}\}\subalign{&0\leq j\leq
N_T\\&l\in{\rm I}}$ and $\{u_{jl}\}\subalign{&0\leq j\leq N_T\\&l\in {\rm I}}$.
Recall that the Fourier-Taylor representations 
\begin{align*}
   \W&\sim\sum_{j=0}^{N_T}\W_j(\te)s^j,
     & u&\sim\sum_{j=0}^{N_T} u_j(\te)s^j,
\end{align*}
are related to the grid-Taylor ones by
\begin{align*}
   \W_{jl}&=\W_j(\te_l), & u_{jl}&=u_j(\te_l),
\end{align*}
where $\W_{jl},u_{jl}\in\RR^{2n}$, and $\{\te_l\}_{l\in
{\rm I}}$ is an equally spaced grid of $\TT^{d}$ indexed by
${\rm I}$. Let us define the truncated Taylor series of $\W,u$
associated to each grid point as
\begin{align*}
   \W_{l}(s) &:= \sum_{j=0}^{N_T} W_{jl} s^j,
         & 
   u_{l}(s) &:= \sum_{j=0}^{N_T} u_{jl} s^j.
\end{align*}
The Jet Transport technique allows the propagation of
$W_{l}$ and $u_{l}$ through the flow in order to
obtain truncated Taylor series of $\phi_T\circ W$ and
$(\Dif\phi_T\comp \W)u$ associated to each grid point as
\begin{gather*}
   \phi_T\comp \W_{l}\sim \sum_{j=0}^{N_T}\phi_{jl}s^j,\\
\left(\Dif\phi_T\comp\W_{l}\right) u_{l}\sim 
\sum_{j=0}^{N_T}\Phi_{jl}s^j,
\end{gather*}
for all $l\in{\rm I}$. Therefore, we obtain
$\{\phi_{jl}\}\subalign{&0\leq j\leq N_T\\&l\in{\rm I}}$ and 
$\{\Phi_{jl}\}\subalign{&0\leq j\leq N_T\\&l\in{\rm I}}$.
That is, we obtain grid-Taylor representations for
\eqref{eq:JT1m} and \eqref{eq:JT2m}. We can recover a
Fourier-Taylor representation with the inverse DFT.

\subsection{Comments on implementations}\label{sec:comments-whiskers}
Throughout this section, assume
$d=1$---which is the case for the numerical experiments of
Section \ref{sec:numerics}---and that we have the following
Fourier-Taylor representation
\begin{equation}\label{eq:tuncW}
   \W(\te,s)\sim\sum_{j=0}^{N_T}\W_j(\te)s^j\sim
   \sum_{j=0}^{N_T}\sum_{k\in\cI} \hat \W_{jk}
   e^{\bi2\pi k\te}s^j,
\end{equation}
for some finite order $N_T$ and a finite multi-index set
$\cI\subset\ZZ^{d}$.

In the following we describe some elementary techniques 
we have found useful. 

\subsubsection{Low pass filter} 
We follow \cite{HM21,FMHM24} and implement
a low-pass filter after each iterative step of Algorithm
\ref{algo:newton}. This is, for some filtering factor
$r_f\in[\frac{1}{4},\frac{1}{2})$, we set to zero the
coefficients $\hat\W_{jk}$ for
$\abs{k}>r_f\cdot N_F$.
In order to decide the number of coefficients for the Fourier
series, i.e., $N_F$, we examine the exponential decay 
since our objects are real analytic. 
In particular, for any $\W_i$, we compute
the tail as
\[
   t(W_j)=\sum_{\abs{k}>r_t\cdot N_F}\|\hat W_{jk} \|_{\infty},
\]
where $r_t<r_f$ is some tail factor. Then, we can decide to
increase $N_F$ if we detect the tails do not decrease
sufficiently fast. See algorithm \ref{algo:implementation}.

\subsubsection{Choice of scale in the $s$ parameter} 
According to Remark \ref{re:scaling}, we can scale
the series in the coordinate $s$ by its estimated radius of
convergence for the sake of normalization. A possibility is to run
algorithm \ref{algo:newton} twice: the first to estimate the
radius of convergence and the second, after the scaling has been
done. However, if we are interested in high order
parameterizations, say up to a maximum order $N_{T_{max}}$, estimating the
radius of convergence after performing algorithm
\ref{algo:newton} to order $N_{T_{max}}$ might not give good
estimates. For this reason, we run algorithm \ref{algo:newton} 
first to order $N_{T_i}<N_{T_{max}}$. Then we estimate the radius of convergence
of the parameterization and scale the series. Then, we iterate at
order $N_{T_{i+1}}>N_{T_i}$ until the radius of convergence is close to 1
or until $N_{T_i}=N_{T_{max}}$. See algorithm \ref{algo:implementation}.

\subsubsection{Multiple Shooting}\label{sec:multiple}
To mitigate the instability of hyperbolicity in the propagation, 
 we follow
\cite{HM21,FMHM24} and implement a multiple-shooting strategy.
For some positive integer $m$,
we look for whiskered tori $\{\cW^i\}_{i=0}^{m-1}$ and
parameterizations $\{\W^i\}_{i=0}^{m-1}$, where
$\W^i:\TT^d\times\RR\to U$, such that the following system of
functional equations is satisfied
\begin{equation}\label{eq:multiple-shooting}
   \phi_{T_m}\comp\W^i-\W^{i+1}\comp\Ro_{\om_m}^{\la_m}=0,\quad
   \text{for}~ i=0,\dots,m-1,
\end{equation}
where the superscript of $W$ is defined modulo $m$, $T_m:=T/m,~
\om_m:=\om/m,$ and $\la_m:=\la^{1/m}$. 
Assuming $\abs{\la}<1$, we choose $m$ 
such that the expansion rate $\abs{\la}^{\frac{-1}{m}}$ is small
enough. For the sake of simplicity, we have described the methodology
for $m=1$---the generalization to $m>1$, albeit tedious,
presents no additional complications.  

\subsubsection{Implementation of jet transport} 
Our Jet Transport implementation substitutes the floating point
operations of a {\tt RKF78} integrator by a simpler yet similar
version of the truncated formal power series manipulator
described in Chapter 2 of \cite{mamotreto}; which is publicly
available. It does so by means of a {\tt C++} wrapper that overloads
the arithmetic operators of a {\tt RKF78 C} routine with jet
operations.  A subtle point is that the step control of the {\tt
RKF78} algorithm involves propagation by two methods and
comparing the two results. The comparison of two polynomials
requires choosing a norm in the space of polynomials. It seems to
be not a critical choice.  In our implementation, the step size
is Fehlbergh's but taking into account all the coefficients of
the jet using the $\ell_1$ norm. This is justified by the fact
that, assuming analyticity, a calculation shows that the error in
all the coefficients of the jet is of the order of the numerical
integrator in the step size.  Reference \cite{GimenoJT} shows
actually more: the jet of coefficients obtained through, e.g.,
Taylor and Runge-Kutta methods is the same (up to rounding
errors) as the one that would be obtained by integrating the
higher order  variational equations up to the order of the jets.

\subsubsection{Parallelization} 
Jet Transport (which profiling shows is the most time consuming 
part of the algorithm) is very easy to parallelize since the integration of 
an orbit can be made independently of the integration of another. 
Hence, the most computationally intensive part of our algorithm can 
be efficiently parallelized, which we do with {\tt OpenMP}.
Note that the approach generalizes straightforwardly to
overloading to polynomials in two or more variables, which is
what will be needed in the computation of higher  dimensional
whiskers. 

\subsubsection{Specification of algorithm in pseudocode} 
We conclude this section with an algorithmic description of our
actual implementation of the iterative scheme of Section
\ref{sec:iter} (Algorithm
\ref{algo:implementation}).
\begin{algorithm}
\caption{Practical implementation}\label{algo:implementation}
\raggedright
Let $\la$ be some approximate rate of contraction and
let $\W_0$ and $\W_1$ be zero and first order terms of $\W$
satisfying \eqref{eq:invW} approximately 
for some fixed $T$ and rotation vector $\om$. Let $\ep_t,
~\ep_W,~\tilde\ep$ be tolerances, $n_t,~n_{it}$ maximum number
of iterates, $\tilde r\in(0,1)$ a threshold for the radius of
convergence, $\De N_T$ an integer increment, $N_{T_i}$ an initial order
for scaling, $N_T$ a desired order, and $N_{F_{max}}$a maximum
number of Fourier coefficients.  Compute $\la$ and $\W$ at order
$N_T$ by following these steps:
\vspace{-1em}
\begin{algorithmic}[1]
   \STATE Compute the tails $t(\W_0)$ and $t(\W_1)$. If 
      $\Max{t(\W_0),\, t(\W_1)}>\ep_t$ set $N_F\leftarrow
      \Min{2N_F,\, N_{F_{max}}}$. Try up to $n_t$ times.
   \STATE \label{it:scale}
   Perform Algorithm \ref{algo:newton} at order $N_{T_i}$ until
   $\norm{E_{N_{T_i}}}<\ep_W$ or up to $n_{it}$ times. 
   \STATE Estimate the radius of convergence
   $r\approx\norm{E_{N_{T_i}}}^{-1/N_{T_i}}$.
   \STATE If $r<\tilde r$ or $r>1$, scale the series for $\W$ by
      $1/r$. Else, go to step \ref{it:newtonNmax}.
      \STATE If $N_{T_i}=N_T$, go to step \ref{it:newtonNmax}.
      Else, set $N_{T_i}\leftarrow \Min{N_{T_i}+\De N_T,\, N_T}$
      and go to step \ref{it:scale}.
   \STATE \label{it:newtonNmax}
      Perform Algorithm \ref{algo:newton} at order $N_T$
      until $\norm{E_{N_T}}<\tilde\ep$ or up to $n_{it}$ times.
      Obtain $\W$ at order $N_T$ and $\la$.
%   \item 
%      If the obtained Choose a truncation criterion, like \eqref{eq:Trunc} or
%      \eqref{eq:Trunc2}, to obtain  a suitable truncation order
%      $N$ and the truncated series for $\W$.
\end{algorithmic}
\end{algorithm}
\begin{remark}
In practice, since the number of correct terms after $it$
iterations is smaller than $2^{it+1}-1$, see Remark
\ref{rem:correct-terms}, we perform iterates at order
$\min\{2^{it+1}-1,N_{T_i}\}$.
\end{remark}

\section{Numerical experiments}\label{sec:numerics}
We apply our algorithms in the Hamiltonian vector field of the
Circular Restricted Three Body Problem (CRTBP). That is, we have
a Hamiltonian system describing the motion of a massless particle
under the influence of the gravitational attraction of of two
massive bodies, known as primaries, moving in circular orbits
around their common barycentre. Let $m_1$ and $m_2$ denote the
masses of the primaries and let us also define the mass parameter
$\mu:=\frac{m_2}{m_1+m_2}$. We can define a rotating frame, also
known as a synodic frame, such that the position of the primaries
is fixed in $x_1$ axis, see e.g. \cite{Szebehely67}. For a
suitable rescaling in space and time, the coordinates of the
primaries are then $(\mu,0,0)$ and $(\mu-1,0,0)$ and their period
of revolution is $2\pi$. The motion of the infinitesimal body is
then described by autonomous Hamiltonian 
\[
  H(x,p)=\frac{1}{2}(p_1^2+p_2^2+p_3^2)-x_1p_2+x_2p_1 - \frac{1-\mu}{r_1}
  -\frac{\mu}{r_2},
\]
where $x,p\in\RR^3$, $r_1^2:=(x_1-\mu)^2+x_2^2+x_3^2$, and
$r_2^2:=(x_1-\mu+1)^2+x_2^2+x_3^2$.

The CRTBP has five fixed points denoted by $L_i$ with
$i=1,2,\dots,5$. We focus on the $L_1$ point which is the one
located between the primaries. In terms of stability, the $L_1$
point is of type center$\times$center$\times$saddle. That is, 
\[
   {\rm Spec} {\rm D}X(L_1) = \{
      \bi 2\pi\omega_p^0, -\bi2\pi\omega_p^0,
      \bi 2\pi\omega_v^0, -\bi2\pi\omega_v^0,
      \lambda^0, -\lambda^0
   \},
\]
for $\omega_p^0,\omega_v^0,\lambda^0\in\RR^+$. Consequently, $L_1$
has a $4-$dimensional center manifold. Lyapunov's center theorem,
see \cite{meyer-hall-offin,siegel-moser}, guarantees the
existence of two families of periodic orbits contained in the
center manifold: the vertical and the planar Lyapunov families.
Note that there exist several other families of periodic orbits,
see e.g. \cite{2003aDiDoPa}, but we limit our case to the
classical Lyapunov families. 
The center manifold carries quasi-periodic motion defined by
$2-$dimensional invariant tori around periodic orbits with a
central part. The quasi-periodic orbits around the Lyapunov
orbits are commonly known as Lissajous orbits---these orbits
correspond to partially hyperbolic invariant tori with rank$-1$
whiskers which will be our case study. More specifically, we
focus on the family of Lissajous tori around the vertical
Lyapunov family.

\subsection{The Lissajous family}\label{sec:KWHM21}
In order to identify each torus of the Lissajous family, we
follow \cite{GomezM01,HM21} and define the rotation number,
denoted by $\rho$, as follows.  Let
$\hat\omega=(\omega_p,\omega_v)$ denote the frequencies of any
$2$-dimensional torus $\hat\cK$ in the Lissajous family
and let $h$ denote the value of the Hamiltonian for any point
$p\in\hat\cK$.
If we denote by $h^0$ the Hamiltonian of the $L_1$ point, we take
$\hat\omega=(\omega_p,\omega_v)$ such that
$\omega_p\rightarrow\omega_p^0$ and
$\omega_v\rightarrow\omega_v^0$ when $h\rightarrow h^0$. Then, we
define the rotation number as
\[
	\rho:=\frac{\omega_p}{\omega_v} -1.
\]
Each torus of the Lissajous family can be uniquely identified by
their value of $h$ and $\rho$. 

For the computation of the Lissajous family of whiskered tori we
have two main options: to perform continuations of families for
fixed $\rho$ following Section \ref{sec:cont} or to compute tori
from approximate solutions of \eqref{eq:invW} following the
iterative scheme of Section \ref{sec:iter}. For the later
approach, we need to have approximate parameterizations. To this
end, we follow the approach of \cite{HM21} and compute
parameterizations of invariant tori and their invariant bundles
for the Lissajous family---that is, we have the terms $W_0$
and $W_1$ the Fourier-Taylor expansion of $\W$.
The tori obtained are included in
Figure \ref{fig:KW}, where the color bar labels the
number of Fourier coefficients used for the terms
$W_0$ and $W_1$.

\begin{figure}[htbp]
\includegraphics{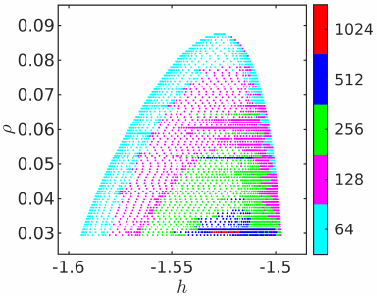}
  \caption{Energy-rotation number representation of the Lissajous tori
  around the $L_1$ point computed following the approach of \cite{HM21}.
  The colorbar labels the number of Fourier coefficients of the
  parameterizations.}
  \label{fig:KW}
\end{figure}
\begin{remark}
   We emphasize that albeit in Figure \ref{fig:KW} we represent
   the total number of Fourier coefficients used to compute tori
   and their bundles, the effective number of coefficients is
   smaller because of the low-pass filter implemented to prevent
   numerical instabilities, see Section
   \ref{sec:comments-whiskers}.
\end{remark}

\subsubsection{Computation of Lissajous whiskers}\label{eq:numerics-whisker}
In this section, we apply Algorithm \ref{algo:implementation}
to each of the 8684 tori obtained for Figure \ref{fig:KW}. Recall that in
the previous section, we obtained the terms $W_0$ and $W_1$ that
we can use as seed for the iterative procedure.
We first compute parameterizations up to order $N_T=10, 12, 16,
20, 50$ with $\ep_t=10^{-9},~\ep_W=\tilde\ep=10^{-4},~n_{t}=1,
~n_{it}=7,~\tilde r=0.9,~\De N_T=10$, and $N_{F_{max}}=1024$. 
See Algorithm \ref{algo:implementation}. For the
multiple-shooting strategy, we take $m=4$, see Section
\ref{sec:multiple}.

In order to assess the quality of the parameterizations, 
we construct an observable $d_{max}$ to measure pointwise the 
maximum distance in phase space (up to the discretization) between tori
and the limit of validity of their whiskers parameterizations
in the corresponding fiber. That is, we consider
\begin{equation}
   \label{eq:dist}
   d_{max}=\max_{\subalign{l&\in {\rm I}\\ 0\leq &j\leq p}}
   \|\W(\theta_l,s_j)-\W(\theta_l,0)\|_2,
\end{equation}
where $\{\te\}_{l\in{\rm I}}$ is an equally spaced grid of $\TT^d$, 
${\rm I}$ is some finite multi-index set with elements in
$\NN^d$, $p$ is some positive integer,  $s_0=-D, s_p=D$, 
$s_j<s_{j+1}$, and $\cD = [-D, D]$ the fundamental domain of
$\W$. For the estimates of the fundamental domain, we tested the
two methods described in Section \ref{sec:domain}. That is, using
the upper bound \eqref{eq:estD} and by requiring \eqref{eq:tol}
to hold using \eqref{eq:discnormE}, for $\ep=10^{-7}$ in both
cases. We obtained very similar results, being the estimate provided by 
\eqref{eq:estD} slightly larger. This shows that the main source
of error is the truncation error, which we cannot control except by increasing the order. 
In what follows, we use criterion \eqref{eq:tol} with the
approximation \eqref{eq:discnormE} to estimate the fundamental
domain.

We include in Figure \ref{fig:dist-fix-orders} the observable
$d_{max}$ in the energy-rotation number representation for
different orders $N_T$.
We have also computed the observable $d_{max}$ when we use the linear
approximation of the whiskers, i.e., when $N_T=1$. The results are
not included in the Figure since in this case
\(d_{max}\) is \(\cO(10^{-4})\). Furthermore, we also computed
the observable $d_{max}$ projecting $\W$ into configuration
space. For $N_T=30$, for instance, we obtain typical values around
$0.12$ and top values around $0.14$. We emphasize that $d_{max}$
reflects the validity of the parameterizations in phase space
whereas $D$ reflects the validity in the coordinates of the
parameterizations---which depend on the scaling of $W_1$.
\begin{figure}[htbp]
\includegraphics{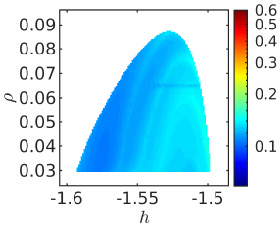}
\includegraphics{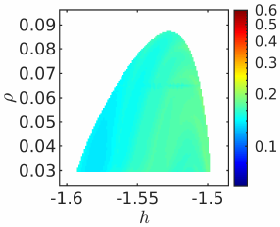}
\includegraphics{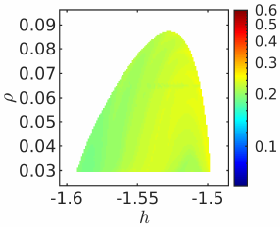}
\includegraphics{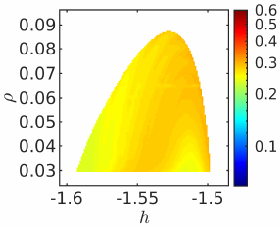}
\includegraphics{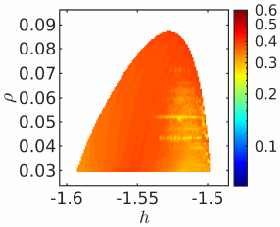}
\includegraphics{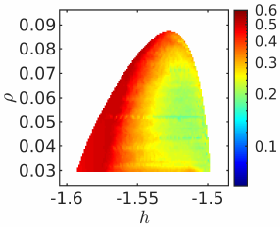}
\caption{Observable $d_{max}$ for
      fixed orders (from left to right) $N_T=10, 12, 16, 20, 30,
   50$. Colorbar in logarithmic scale.} \label{fig:dist-fix-orders}
\end{figure}
We can observe that increasing the order of the parameterizations
from $N_T=10$ to $N_T=30$ results in larger domains for almost all tori.
However, if we increase the order to $N_T=50$, there are tori for which we
obtain larger domains but, in certain region of the
energy-rotation number representation, we obtain worse
parameterizations. In other words, increasing the order of the
parameterizations does not always result in larger fundamental
domains. It is then clear that, due to numerical instabilities,
if we want to maximize the fundamental domain for each tori, we
need to choose suitable orders for each parameterization.

Let us now explore individual truncation orders and set a maximum
order for the parameterizations $N_{T_{max}}=50$.  For the tori
obtained  at order $50$ in Figure \ref{fig:dist-fix-orders}, we
truncate the series as follows. For each torus,
we consider Fourier-Taylor representations for the error,
defined in \eqref{eq:error}, at order $N_{T_{max}}$. That is, we
consider
\[
   E(\te,s) \sim  \sum_{j=0}^{N_{T_{max}}}E_j(\te)s^j.
\]
Then, we truncate the parameterizations for each torus at the
largest order $N_T\leq N_{T_{max}}$ for which
\begin{equation}\label{eq:trunc-crit}
   \norm{E_j}<\tol
\end{equation}
holds for all $j\leq N_T$ and some threshold $\tol>0$. We run tests for
$\tol=10^{-4},~1$ and
include the orders obtained in Figure \ref{fig:maxO} and the
observable $d_{max}$ in Figure \ref{fig:maxD}. Note that in the
truncated series for $E$, the term $E_j$ contributes to the
error as $\norm{E_j} D^j$ where $D<1$.
\begin{figure}[htbp]
\includegraphics{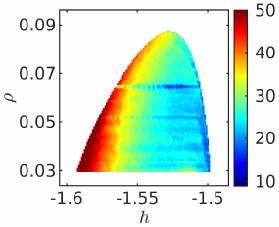}
\includegraphics{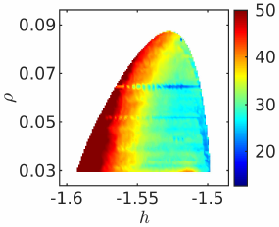}
\caption{Truncation order for $\tol=10^{-4}$ (left) and $\tol=1$
(right).}
\label{fig:maxO}
\end{figure}

\begin{figure}[htbp]
\includegraphics{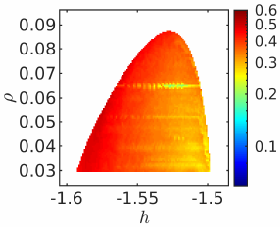}
\includegraphics{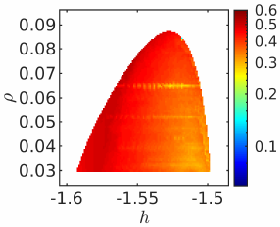}
\caption{Observable $d_{max}$ for $tol=10^{-4}$ (left) and $tol=1$
   (right). Colorbar in logarithmic scale.}
\label{fig:maxD}
\end{figure}
We can observe that, with this approach, we improve the results
obtained in Figure \ref{fig:dist-fix-orders}: for most tori, we
obtain values of the observable $d_{max}$ larger that the
obtained for the fixed order results of Figure
\ref{fig:dist-fix-orders}. This shows that each torus in the
energy-rotation number representation requires different number of
Fourier coefficients and different truncation orders in their
Fourier-Taylor parameterizations. Recall that the number of
Fourier coefficients was determined by the seed that we used for the
computation of whiskers coming from \cite{HM21} and from their
tails, see Algorithm \ref{algo:implementation}.

As a last experiment, we proceed to compute the optimal orders of
the parameterizations via direct computation. More specifically,
we use the parameterizations obtained in Figure
\ref{fig:dist-fix-orders} at fixed order $N_{T_{max}}=50$. Then, 
for some $N_T\leq N_{T_{max}}$ and $j=1,\dots,N_T$, 
we construct increasing sequences for the fundamental domains in phase
space of truncated parameterizations at order $j$ according to
Section \ref{sec:domain}. That is, we construct the
sequences $\{d^j_{max}\}_{1\leq j\leq N_T}$ where $N_T$
is the largest integer for which $\{d^j_{max}\}_{1\leq j\leq N_T}$
is strictly increasing. Then, we take this $N_T$ as the order that
optimizes the fundamental domain of the parameterizations.
The results are included in Figure \ref{fig:optimal-orders-dist},
where we can observe that the results are rather similar to
those obtained in Figures \ref{fig:maxO}-\ref{fig:maxD} using
criterion \eqref{eq:trunc-crit} for $\tol=1$.
Lastly, we include a contour plot in Figure \ref{fig:maxD-cont}
for the observable $d_{max}$ obtained with criterion
\eqref{eq:trunc-crit} and via direct computation. It is then clear that criterion
\eqref{eq:trunc-crit} with $\tol=1$ is a suitable choice for
choosing truncation orders that maximize the fundamental domains
in phase space. It is worth mentioning that the computation of
the optimal orders via direct computation cannot deal with folds
since we are only constructing the sequences
$\{d^j_{max}\}_{1\leq j\leq N_T}$. We do not, in general, expect
folds locally close to the invariant torus. If necessary, one
could instead inspect the computationally more expensive
sequences $\{d^j_{max}\}_{1\leq j\leq N_{T_{max}}}$ if folds are
expected.

\begin{figure}[htbp]
\includegraphics{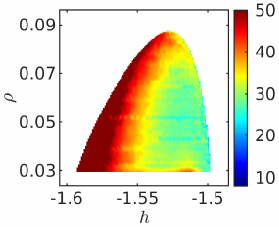}
\includegraphics{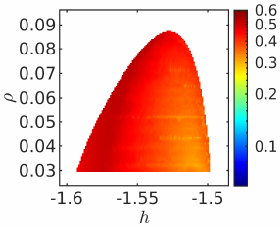}
   \caption{Optimal orders (left) and maximum distance (right)
   obtained via direct computation.}
\label{fig:optimal-orders-dist}
\end{figure}

\begin{figure}[htbp]
\includegraphics{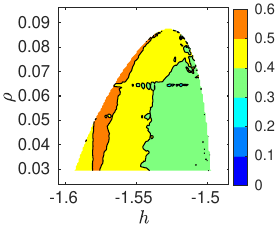}
\includegraphics{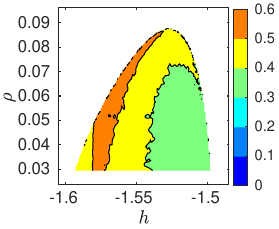}
\caption{Observable $d_{max}$ obtained from criterion
   \eqref{eq:trunc-crit} for $tol=1$ (left) and optimal distance
   obtained via direct computation(right).}
\label{fig:maxD-cont}
\end{figure}

We end this section with an illustration of the computing times
of the procedure. We run Algorithm \ref{algo:implementation} but
without scaling the bundle, i.e., we run only the steps $1)$ and
$2)$ at $N_{T_i} = N_T$. We take the torus labeled by $h = -1.5319$ and
$\rho = 0.0723$ and compute its stable whisker for different
numbers of Fourier coefficients and at different orders, see
Table \ref{tab:comp-times}. We include the error of
invariance in the torus ($\|E_0\|$), the bundle ($\|E_1\|$), the last term of the
Fourier-Taylor representation of $E$ ($\|E_{N_T}\|$), the limit of the fundamental
domain in the coordinates of the parameterization (denoted by
$D$), the observable $d_{max}$, the number of iterative steps,
and the computing time.  The computations are done in
a machine under Debian 12 with two processors Intel(R) Xeon(R)
CPU E5-2630 0 @ 2.30GHz with 6 cores each and two threads per
core for a total of 24 threads.

We first observe that when for $N_F=64$, we do not have enough Fourier
coefficients to obtain parameterizations satisfying the invariance equation
with an error smaller than $\cO(10^{-8})$. Also, note that while
$\norm{E_{N_T}}$
grows with $N_T$, its contribution to the error is
$\norm{E_{N_T}}D^{N_T}$.  Therefore,
large values of $\norm{E_{N_T}}$ do not result in large errors. 
It simply means 
that $D$ needs to be sufficiently small to guarantee an accurate description of
the parameterizations in the fundamental domain of $\W$. Also, we have observed
that the computational time grows linearly with $N_F$. This is not surprising
since a profiling of the code reveals that $98.4 \%$ of the total time is
numerical integration. Regarding the computing time behaviour with the order,
from an operation count standpoint, a quadratic increase of computing time with
order should be expected. Profiling tests have also shown that quadratic growth
is not observed in Table \ref{tab:comp-times} (indeed, the growth of 
the time seems less than linear in $N_T$)  because of the computational
effort required by the book-keeping and memory management related to the
manipulation of power series in the Jet Transport implementation.

\def\arraystretch{1.5}   
\begin{table}[ht!]
\centering
\caption{Computational times and errors}
\label{tab:comp-times}
\begin{tabular}{c c c c c c c c c}\hline
   $N_F$ & $N_T$ &$\norm{E_0}$ &$\norm{E_1}$ & $\norm{E_{N_T}}$& $D$ &
$d_{max}$ & it & $t$\\
\hline\hline
\multirow{6}{1em}{$64$} & $10$ & $2.3\cdot10^{-8}$ &
$6.5\cdot10^{-8}$ & $2.6\cdot 10^{-6}$ & $0.333$ & $0.122$ &
$5$ & $2.68$s \\
 & $12$ & $2.3\cdot10^{-8}$ & $6.5\cdot10^{-8}$ &
$3.5\cdot10^{-6}$ & $0.424$ & $0.162$ & $5$ & $2.76$s \\
& $16$ & $2.3\cdot10^{-8}$ & $6.5\cdot10^{-8}$ &
$2.8\cdot10^{-5}$ & $0.534$ & $0.214$ & $5$ & $3.05$s \\
& $20$ & $2.3\cdot10^{-8}$ & $6.5\cdot10^{-8}$ &
$4.7\cdot10^{-6}$ & $0.567$ & $0.231$ & $6$ & $3.85$s \\
& $30$ & $2.3\cdot10^{-8}$ & $6.5\cdot10^{-8}$ &
$3.4\cdot10^{-1}$ & $0.552$ & $0.223$ & $7$ & $5.49$s \\
& $50$ & $2.3\cdot10^{-8}$ & $6.5\cdot10^{-8}$ &
$2.0\cdot10^9$ & $0.457$ & $0.177$ & $7$ & $7.15$s \\\hline
\multirow{6}{1em}{$256$} & $10$ & $1.4\cdot10^{-14}$ &
$1.1\cdot10^{-14}$ & $1.1\cdot10^{-13}$ & $0.335$ & $0.123$ &
$5$ & $9.90$s \\
 & $12$ & $1.4\cdot10^{-14}$ & $1.1\cdot10^{-14}$ &
$1.3\cdot10^{-12}$ & $0.427$ & $0.163$ & $5$ & $10.25$s \\
 & $16$ & $1.4\cdot10^{-14}$ & $1.1\cdot10^{-14}$ &
$1.2\cdot10^{-9}$ & $0.557$ & $0.226$ & $5$ & $11.02$s \\
& $20$ & $1.6\cdot10^{-14}$ & $1.3\cdot10^{-14}$ &
$1.0\cdot10^{-13}$ & $0.657$ & $0.279$ & $6$ & $14.25$s \\
& $30$ & $1.6\cdot10^{-14}$ & $1.3\cdot10^{-14}$ &
$1.4\cdot10^{-8}$ & $0.839$ & $0.383$ & $6$ & $16.02$s \\
& $50$ & $1.6\cdot10^{-14}$ & $1.3\cdot10^{-14}$ &
$3.7\cdot10^{-6}$ & $0.692$ & $0.298$ & $6$ & $19.30$s\\\hline
\end{tabular}
\end{table}

%\multirow{6}{2em}{$1024$} & $10$ & $1.5\cdot10^{-14}$ &
%$1.3\cdot10^{-14}$ & $1.6\cdot10^{-12}$ & $0.335$ & $0.123$ &
%$5$ & $37.99$s \\
%& $12$ & $1.5\cdot10^{-14}$ & $1.3\cdot10^{-14}$ &
%$2.0\cdot10^{-12}$ & $0.427$ & $0.163$ & $5$ & $39.60$s \\
% & $16$ & $1.5\cdot10^{-14}$ & $1.3\cdot10^{-14}$ &
%$1.3\cdot10^{-8}$ & $0.557$ & $0.226$ & $5$ & $42.07$s \\
%& $20$ & $1.4\cdot10^{-14}$ & $1.3\cdot10^{-14}$ &
%$1.4\cdot10^{-13}$ & $0.657$ & $0.279 $ & $6$ & $54.78$s \\
%& $30$ & $1.4\cdot10^{-14}$ & $1.3\cdot10^{-14}$ &
%$1.5\cdot10^{-8}$ & $0.839$ & $0.383$ & $6$ & $62.88$s \\
%& $50$ & $1.4\cdot10^{-14}$ & $1.3\cdot10^{-14}$ &
%$3.8\cdot10^{-6}$ & $0.646$ & $0.273$ & $6$ & $76.07$s \\
%\hline

\bibliographystyle{alpha}
\bibliography{references} % references.bib

\end{document}

%% file: macros.tex
\newcommand{\al}{\alpha}
\newcommand{\be}{\beta}
\newcommand{\de}{\delta}
\newcommand{\ep}{\varepsilon}

\renewcommand\epsilon{\varepsilon}

\newcommand{\la}{\lambda}
\newcommand{\om}{\omega}

\newcommand{\te}{\theta}

\newcommand{\De}{\Delta}

\newcommand{\La}{\Lambda}

\newcommand{\Om}{\Omega}

\newcommand{\cA}{{\mathcal A}}

\newcommand{\cD}{{\mathcal D}}

\newcommand{\cI}{{\mathcal I}}

\newcommand{\cK}{{\mathcal{K}}}

\newcommand{\cO}{{\mathcal O}}

\newcommand{\cS}{{\mathcal S}}

\newcommand{\cW}{{\mathcal W}}

\newcommand{\RR}{{\mathbb R}}
\newcommand{\CC}{{\mathbb C}}

\newcommand{\TT}{{\mathbb T}}
\newcommand{\ZZ}{{\mathbb Z}}
\newcommand{\NN}{{\mathbb N}}

\newcommand\bi{\boldsymbol{i}}
\newcommand\bJ{\boldsymbol{J}}
\newcommand\bg{\boldsymbol{g}}
\newcommand{\abs}[1]{|{#1}|}

\newcommand{\norm}[1]{\|{#1}\|}

\newcommand{\aver}[1]{{\langle{#1}\rangle}}
\newcommand{\Max}[1]{{\max{\left\{#1\right\}}}}
\newcommand{\Min}[1]{{\min{\left\{#1\right\}}}}

\newcommand{\phiT}{\phi_T}
\newcommand{\DzphiT}{{\Dif_z}\phi_T}

\newcommand\ee{\mathrm{e}}

\newcommand\dif{{\rm d}}
\newcommand\Dif{{\rm D}}

\newcommand\pd{ \partial}

\newcommand\Ro{{R}}
\newcommand\Rom{R_\om}
\newcommand\Rw{{R_\om^\la}\hspace{1pt}}
\newcommand{\comp}{{\!\:\circ\!\:}}

\newcommand\ttop{{\!\top\!}}

\newcommand\K{K}

\newcommand\W{{W}}

\renewcommand\L{{L}}

\newcommand\N{{N}}

\renewcommand\S{\cS}

\renewcommand\P{{P}}
\newcommand\PT{{P^\ttop}}

\newcommand\G{G}

\newcommand\GL{G_\L}

\newcommand\barla{\bar\la}

\newcommand\Dela{{\De \la}}
\newcommand\DeW{{\De \W}}

\newcommand\bOm{\boldsymbol{\om}}

\newcommand\J{{J}}

\newcommand\OmO{{\Omega_{0}}}

\newcommand\OmL{{\Omega_{L}}}

%% file: main_7.bbl
\newcommand{\etalchar}[1]{$^{#1}$}
\def\cprime{$'$} \def\cprime{$'$} \def\cprime{$'$} \def\cprime{$'$}
\begin{thebibliography}{FMHM24b}

\bibitem[AA68]{ArnoldA68}
V.~I. Arnold and A.~Avez.
\newblock {\em Ergodic problems of classical mechanics}.
\newblock W. A. Benjamin, Inc., New York-Amsterdam, 1968.
\newblock Translated from the French by A. Avez.

\bibitem[AFV{\etalchar{+}}09]{jetESA}
E.~M. Alessi, A.~Farr{\'e}s, A.~Vieiro, {\`A}.~Jorba, and C.~Sim{\'o}.
\newblock Jet transport and applications to neos.
\newblock In {\em Proceedings of the 1st IAA Planetary Defense Conference,
  Granada, Spain}, pages 10--11, 2009.

\bibitem[And21]{Anderson21}
R.~L. Anderson.
\newblock Tour design using resonant-orbit invariant manifolds in patched
  circular restricted three-body problems.
\newblock {\em Journal of Guidance, Control, and Dynamics}, 44:106–119, 2021.

\bibitem[Arn63]{Arnold63}
V.~I. Arnol.
\newblock Small denominators and problems of stability of motion in classical
  and celestial mechanics.
\newblock {\em Uspehi Mat. Nauk}, 18(6(114)):91--192, 1963.

\bibitem[Arn64]{Arnold64}
V.~I. Arnold.
\newblock Instability of dynamical systems with many degrees of freedom.
\newblock {\em Dokl. Akad. Nauk SSSR}, 156:9--12, 1964.

\bibitem[BB20]{BonaseraB20}
S.~Bonasera and N.~Bosanac.
\newblock Transitions between quasi-periodic orbits near resonances in the
  circular restricted three-body problem.
\newblock In {\em Proceedings of the AAS/AIAA Astrodynamics Specialist Virtual
  Conference}. AAS, 2020.

\bibitem[Ber01]{Berndt01}
R.~Berndt.
\newblock {\em An introduction to symplectic geometry}, volume~26 of {\em
  Graduate Studies in Mathematics}.
\newblock American Mathematical Society, Providence, RI, 2001.
\newblock Translated from the 1998 German original by Michael Klucznik.

\bibitem[BHM24]{Barcelona24}
M.~Barcelona, {\`A}.~Haro, and J.~M. Mondelo.
\newblock Semianalytical computation of heteroclinic connections between center
  manifolds with the parameterization method.
\newblock {\em SIAM Journal on Applied Dynamical Systems}, 23(1):98--126, 2024.

\bibitem[BMO09]{2009BaMoOll}
E.~Barrab{\'e}s, J.~M. Mondelo, and M.~Oll{\'e}.
\newblock Numerical continuation of families of homoclinic connections of
  periodic orbits in the {RTBP}.
\newblock {\em Nonlinearity}, 22(12):2901--2918, 2009.

\bibitem[BMO13]{BarrabesMO13}
E.~Barrab{\'e}s, J.~M. Mondelo, and M.~Oll{\'e}.
\newblock Numerical continuation of families of heteroclinic connections
  between periodic orbits in a {H}amiltonian system.
\newblock {\em Nonlinearity}, 26(10):2747, 2013.

\bibitem[Bol90]{Bolotin90}
S.~V. Bolotin.
\newblock Motions that are doubly asymptotic to invariant tori in the theory of
  the perturbations of {H}amiltonian systems.
\newblock {\em Prikl. Mat. Mekh.}, 54(3):497--502, 1990.

\bibitem[Bol95]{Bolotin95}
S.~V. Bolotin.
\newblock Homoclinic orbits in invariant tori of {H}amiltonian systems.
\newblock In {\em Dynamical systems in classical mechanics}, volume 168 of {\em
  Amer. Math. Soc. Transl. Ser. 2}, pages 21--90. Amer. Math. Soc., Providence,
  RI, 1995.

\bibitem[CDH{\etalchar{+}}12]{CallejaDHLO12}
R.~C. Calleja, E.~J. Doedel, A.~R. Humphries, A.~Lemus-Rodríguez, and E.~B.
  Oldeman.
\newblock Boundary-value problem formulations for computing invariant manifolds
  and connecting orbits in the circular restricted three body problem.
\newblock {\em Celestial Mechanics and Dynamical Astronomy}, 114:77–106,
  2012.

\bibitem[CdS01]{Cannas01}
A.~Cannas~da Silva.
\newblock {\em Lectures on symplectic geometry}, volume 1764 of {\em Lecture
  Notes in Mathematics}.
\newblock Springer-Verlag, Berlin, 2001.

\bibitem[CFH25]{CaraccioloFH25}
C.~Caracciolo, J.-Ll Figueras, and {\`A}.~Haro.
\newblock A parametrization algorithm to compute lower dimensional elliptic
  tori in {H}amiltonian systems.
\newblock {\em Nonlinearity}, 38(2):Paper No. 025005, 25, 2025.

\bibitem[CH17a]{CanadellH17b}
M.~Canadell and \`A. Haro.
\newblock Computation of {Q}uasi-{P}eriodic {N}ormally {H}yperbolic {I}nvariant
  {T}ori: {A}lgorithms, {N}umerical {E}xplorations and {M}echanisms of
  {B}reakdown.
\newblock {\em J. Nonlinear Sci.}, 27(6):1829--1868, 2017.

\bibitem[CH17b]{CanadellH17a}
M.~Canadell and {\`A}.~Haro.
\newblock Computation of {Q}uasiperiodic {N}ormally {H}yperbolic {I}nvariant
  {T}ori: {R}igorous {R}esults.
\newblock {\em J. Nonlinear Sci.}, 27(6):1869--1904, 2017.

\bibitem[CM08]{CanaliasM08}
E.~Canalias and J.~J. Masdemont.
\newblock {Computing natural transfers between Sun-Earth and Earth-Moon
  Lissajous libration point orbits}.
\newblock {\em {Acta Astronautica}}, {63}({1-4}):{238--248}, {2008}.

\bibitem[DDP03]{2003aDiDoPa}
D.~J. Dichmann, E.~J. Doedel, and R.~C. Paffenroth.
\newblock The computation of periodic solutions of the 3-body problem using the
  numerical continuation software {AUTO}.
\newblock In G.~G{\'o}mez, M.~W. Lo, and J.~J. Masdemont, editors, {\em
  Libration Point Orbits and Applications}, pages 489--528. World Scientific,
  2003.

\bibitem[Eli94]{Eliasson94}
L.~H. Eliasson.
\newblock Biasymptotic solutions of perturbed integrable {H}amiltonian systems.
\newblock {\em Bol. Soc. Brasil. Mat. (N.S.)}, 25(1):57--76, 1994.

\bibitem[FdlLS09]{FontichLS09}
E.~Fontich, R.~de~la Llave, and Y.~Sire.
\newblock Construction of invariant whiskered tori by a parameterization
  method. {I}. {M}aps and flows in finite dimensions.
\newblock {\em J. Differential Equations}, 246(8):3136--3213, 2009.

\bibitem[FM00]{FontichM00}
E.~Fontich and P.~Mart\'in.
\newblock Differentiable invariant manifolds for partially hyperbolic tori and
  a lambda lemma.
\newblock {\em Nonlinearity}, 13(5):1561--1593, 2000.

\bibitem[FMHM24a]{FMHM24}
\'{A}. Fern\'{a}ndez-Mora, {\`A}.~Haro, and J.~M. Mondelo.
\newblock Flow map parameterization methods for invariant tori in
  quasi-periodic {H}amiltonian systems.
\newblock {\em SIAM Journal on Applied Dynamical Systems}, 23(1):127--166,
  2024.

\bibitem[FMHM24b]{FMHMKAM}
{\'A}.~Fern{\'a}ndez-Mora, {\`A}.~Haro, and J.~M. Mondelo.
\newblock On the convergence of flow map parameterization methods for whiskered
  tori in quasi-periodic {H}amiltonian systems.
\newblock {\em arXiv preprint arXiv:2411.11772}, 2024.

\bibitem[GJJC{\etalchar{+}}23]{GimenoJT}
J.~Gimeno, {\`A}.~Jorba, M.~Jorba-Cusc{\'o}, N.~Miguel, and M.~Zou.
\newblock Numerical integration of high-order variational equations of odes.
\newblock {\em Applied Mathematics and Computation}, 442:127743, 2023.

\bibitem[GJNO22]{Gimeno22}
J.~Gimeno, \`A Jorba, B.~Nicol{\'a}s, and E.~Olmedo.
\newblock Numerical computation of high-order expansions of invariant manifolds
  of high-dimensional tori.
\newblock {\em SIAM Journal on Applied Dynamical Systems}, 21(3):1832--1861,
  2022.

\bibitem[GJSM01a]{ESA3}
G.~G{\'o}mez, {\`A}.~Jorba, C.~Sim{\'o}, and J.~Masdemont.
\newblock {\em Dynamics and mission design near libration points. {V}ol.
  {III}}, volume~4 of {\em World Scientific Monograph Series in Mathematics}.
\newblock World Scientific Publishing Co. Inc., River Edge, NJ, 2001.
\newblock Advanced methods for collinear points.

\bibitem[GJSM01b]{ESA4}
G.~G{\'o}mez, {\`A}.~Jorba, C.~Sim{\'o}, and J.~Masdemont.
\newblock {\em Dynamics and mission design near libration points. {V}ol. {IV}},
  volume~5 of {\em World Scientific Monograph Series in Mathematics}.
\newblock World Scientific Publishing Co. Inc., River Edge, NJ, 2001.
\newblock Advanced methods for triangular points.

\bibitem[GKL{\etalchar{+}}04]{GomezKLMMR04}
G.~G{\'o}mez, W.~S. Koon, M.~W. Lo, J.~E. Marsden, J.~Masdemont, and S.~D.
  Ross.
\newblock Connecting orbits and invariant manifolds in the spatial restricted
  three-body problem.
\newblock {\em Nonlinearity}, 17(5):1571--1606, 2004.

\bibitem[GLMS01]{ESA1}
G.~G{\'o}mez, J.~Llibre, R.~Mart{\'{\i}}nez, and C.~Sim{\'o}.
\newblock {\em Dynamics and mission design near libration points. {V}ol. {I}},
  volume~2 of {\em World Scientific Monograph Series in Mathematics}.
\newblock World Scientific Publishing Co. Inc., River Edge, NJ, 2001.
\newblock Fundamentals: the case of collinear libration points, With a foreword
  by Walter Flury.

\bibitem[GM01]{GomezM01}
G.~G{\'o}mez and J.~M. Mondelo.
\newblock The dynamics around the collinear equilibrium points of the {RTBP}.
\newblock {\em Phys. D}, 157(4):283--321, 2001.

\bibitem[GMM02]{gomez2002solar}
G.~G{\'o}mez, J.~J. Masdemont, and J.~M. Mondelo.
\newblock Solar system models with a selected set of frequencies.
\newblock {\em Astronomy \& Astrophysics}, 390(2):733--749, 2002.

\bibitem[GR03]{GideaR03}
M.~Gidea and C.~Robinson.
\newblock Topologically crossing heteroclinic connections to invariant tori.
\newblock {\em J. Differential Equations}, 193(1):49--74, 2003.

\bibitem[Gra74]{Graff74}
S.~M. Graff.
\newblock On the conservation of hyperbolic invariant tori for {H}amiltonian
  systems.
\newblock {\em J. Differential Equations}, 15:1--69, 1974.

\bibitem[GSLM01]{ESA2}
G.~G{\'o}mez, C.~Sim{\'o}, J.~Llibre, and R.~Mart{\'{\i}}nez.
\newblock {\em Dynamics and mission design near libration points. {V}ol. {II}},
  volume~3 of {\em World Scientific Monograph Series in Mathematics}.
\newblock World Scientific Publishing Co. Inc., River Edge, NJ, 2001.
\newblock Fundamentals: the case of triangular libration points.

\bibitem[HBW97]{Genesis}
K.~C. Howell, R.~S. Barden, and R.~S. Wilson.
\newblock Trajectory design using a dynamical systems approach with application
  to genesis.
\newblock {\em Adv. Astron. Sci.}, 1997.

\bibitem[HCF{\etalchar{+}}16]{mamotreto}
\`A. Haro, M.~Canadell, J.-Ll. Figueras, A.~Luque, and J.~M. Mondelo.
\newblock {\em The parameterization method for invariant manifolds}, volume 195
  of {\em Applied Mathematical Sciences}.
\newblock Springer, [Cham], 2016.
\newblock From rigorous results to effective computations.

\bibitem[HdlL06a]{HaroL06b}
{\`A}.~Haro and R.~de~la Llave.
\newblock A parameterization method for the computation of invariant tori and
  their whiskers in quasi-periodic maps: numerical algorithms.
\newblock {\em Discrete Contin. Dyn. Syst. Ser. B}, 6(6):1261--1300, 2006.

\bibitem[HdlL06b]{HaroL06a}
{\`A}.~Haro and R.~de~la Llave.
\newblock A parameterization method for the computation of invariant tori and
  their whiskers in quasi-periodic maps: rigorous results.
\newblock {\em J. Differential Equations}, 228(2):530--579, 2006.

\bibitem[HdlL07]{HaroL07}
{\`A}.~Haro and R.~de~la Llave.
\newblock A parameterization method for the computation of invariant tori and
  their whiskers in quasi-periodic maps: explorations and mechanisms for the
  breakdown of hyperbolicity.
\newblock {\em SIAM J. Appl. Dyn. Syst.}, 6(1):142--207 (electronic), 2007.

\bibitem[HdlLS12]{HuguetLS12}
G.~Huguet, R.~de~la Llave, and Y.~Sire.
\newblock Computation of whiskered invariant tori and their associated
  manifolds: new fast algorithms.
\newblock {\em Discrete Contin. Dyn. Syst.}, 32(4):1309--1353, 2012.

\bibitem[Hen79]{Henrici79}
P.~Henrici.
\newblock Fast fourier methods in computational complex analysis.
\newblock {\em SIAM Review}, 21(4):481--527, 1979.

\bibitem[HM21]{HM21}
{\`A}.~Haro and J.~M. Mondelo.
\newblock Flow map parameterization methods for invariant tori in {H}amiltonian
  systems.
\newblock {\em Communications in Nonlinear Science and Numerical Simulation},
  101:105859, 2021.

\bibitem[HPS70]{HirschPS70}
M.~W. Hirsch, C.~C. Pugh, and M.~Shub.
\newblock {Invariant manifolds}.
\newblock {\em Bull. Amer. Math. Soc.}, 76(5):1015--1019, 1970.

\bibitem[HS23]{HenryS23}
D.~B. Henry and D.~J. Scheeres.
\newblock Quasi-periodic orbit transfer design via whisker intersection sets.
\newblock {\em Journal of Guidance, Control, and Dynamics}, 46:1929–1944,
  2023.

\bibitem[JV97]{JorbaV97}
{\`A}.~Jorba and J.~Villanueva.
\newblock On the normal behaviour of partially elliptic lower-dimensional tori
  of {H}amiltonian systems.
\newblock {\em Nonlinearity}, 10(4):783--822, 1997.

\bibitem[KAdlL21a]{KumarAL21}
B.~Kumar, R.~L. Anderson, and R.~de~la Llave.
\newblock High-order resonant orbit manifold expansions for mission design in
  the planar circular restricted 3-body problem.
\newblock {\em Communications in Nonlinear Science and Numerical Simulation},
  97:105691, 2021.

\bibitem[KAdlL21b]{Kumar21GPU}
B.~Kumar, R.~L. Anderson, and R.~de~la Llave.
\newblock Using {GPU}s and the parameterization method for rapid search and
  refinement of connections between tori in periodically perturbed planar
  circular restricted 3-body problems.
\newblock {\em arXiv preprint arXiv}, 2109, 2021.

\bibitem[KAdlL22]{KumarALl22}
B.~Kumar, R.~L. Anderson, and R.~de~la Llave.
\newblock Rapid and accurate methods for computing whiskered tori and their
  manifolds in periodically perturbed planar circular restricted 3-body
  problems.
\newblock {\em Celestial Mech. Dynam. Astronom.}, 134(1):1--38, 2022.

\bibitem[LV11]{LuqueV11}
A.~Luque and J.~Villanueva.
\newblock A {KAM} theorem without action-angle variables for elliptic lower
  dimensional tori.
\newblock {\em Nonlinearity}, 24(4):1033--1080, 2011.

\bibitem[Ma{\~n}78]{Mane78}
R.~Ma{\~n}{\'e}.
\newblock Persistent manifolds are normally hyperbolic.
\newblock {\em Trans. Amer. Math. Soc.}, 246:261--283, 1978.

\bibitem[Mat68]{Mather68}
J.~N. Mather.
\newblock Characterization of {A}nosov diffeomorphisms.
\newblock {\em Nederl. Akad. Wetensch. Proc. Ser. A 71 = Indag. Math.},
  30:479--483, 1968.

\bibitem[MH23]{McCarthyH23}
B.~McCarthy and K.~Howell.
\newblock Construction of heteroclinic connections between quasi-periodic
  orbits in the three-body problem.
\newblock {\em The Journal of the Astronautical Sciences}, 70, 2023.

\bibitem[MHO09]{meyer-hall-offin}
K.~R. Meyer, G.~R. Hall, and D.~Offin.
\newblock {\em {I}ntroduction to {H}amiltonian {D}ynamical {S}ystems and the
  {$N$}--{B}ody {P}roblem}.
\newblock Springer--Verlag, 2nd edition, 2009.

\bibitem[OB24]{OwenB24}
D.~Owen and N.~Baresi.
\newblock Applications of knot theory to the detection of heteroclinic
  connections between quasi-periodic orbits.
\newblock {\em Astrodynamics}, 8:577--595, 2024.

\bibitem[Oli16]{Olikara16}
Z.~P. Olikara.
\newblock {\em Computation of quasi-periodic tori and heteroclinic connections
  in astrodynamics using collocation techniques.}
\newblock PhD thesis, University of Colorado at Boulder, 2016.

\bibitem[Poi87]{Poincare87c}
H.~Poincar{\'e}.
\newblock {\em Les m\'ethodes nouvelles de la m\'ecanique c\'eleste. {T}ome
  {III}}.
\newblock Les Grands Classiques Gauthier-Villars. [Gauthier-Villars Great
  Classics]. Librairie Scientifique et Technique Albert Blanchard, Paris, 1987.

\bibitem[PP15]{tesisPerez}
D.~Pérez-Palau.
\newblock {\em {Dynamical transport mechanisms in celestial mechanics and
  astrodynamics problems}}.
\newblock PhD thesis, Universitat de Barcelona, 2015.

\bibitem[PPMG15]{PerezPMG15}
D.~P\'erez-Palau, J.~J. Masdemont, and G.~G\'omez.
\newblock Tools to detect structures in dynamical systems using jet transport.
\newblock {\em Celestial Mech. Dynam. Astronom.}, 123(3):239--262, 2015.

\bibitem[R{\"u}s75]{Russmann75}
H.~R{\"u}ssmann.
\newblock On optimal estimates for the solutions of linear partial differential
  equations of first order with constant coefficients on the torus.
\newblock In {\em Dynamical systems, theory and applications (Rencontres,
  Battelle Res. Inst., Seattle, Wash., 1974)}, pages 598--624. Lecture Notes in
  Phys., Vol. 38. Springer, Berlin, 1975.

\bibitem[SBA{\etalchar{+}}14]{Artemis}
T.~H. Sweetser, S.~B. Broschart, V.~Angelopoulos, G.~J. Whiffen, D.~C. Folta,
  M.~K. Chung, S.~J. Hatch, and M.~A. Woodard.
\newblock Artemis mission design.
\newblock {\em The ARTEMIS Mission}, pages 61--91, 2014.

\bibitem[SM95]{siegel-moser}
C.~L. Siegel and J.~K. Moser.
\newblock {\em Lectures on celestial mechanics}.
\newblock Classics in Mathematics. Springer-Verlag, Berlin, 1995.
\newblock Translated from German by C. I. Kalme, Reprint of the 1971
  translation.

\bibitem[Sma65]{Smale63}
S.~Smale.
\newblock Diffeomorphisms with many periodic points.
\newblock In {\em Differential and Combinatorial Topology, A Symposium in Honor
  of Marston Morse}, pages 63--80, Princeton, 1965. Princeton University Press.

\bibitem[Sma67]{Smale67}
S.~Smale.
\newblock Differentiable dynamical systems.
\newblock {\em Bulletin of the American mathematical Society}, 73(6):747--817,
  1967.

\bibitem[SPH24]{UTSHowell25}
R.~R. Sanaga, B.~Park, and K.~C. Howell.
\newblock Unified transition scheme for invariant tori across various models in
  the cislunar domain.
\newblock \url{https://ssrn.com/abstract=5271908}, 2024.
\newblock Available at SSRN: \url{https://ssrn.com/abstract=5271908}.

\bibitem[SS74]{SackerS74}
R.~J. Sacker and G.~R. Sell.
\newblock Existence of dichotomies and invariant splittings for linear
  differential systems. {I}.
\newblock {\em J. Differential Equations}, 15:429--458, 1974.

\bibitem[Sze67]{Szebehely67}
V.~Szebehely.
\newblock {\em Theory of orbits. The {R}estricted {P}roblem of {T}hree
  {B}odies}.
\newblock Academic {P}ress, 1967.

\bibitem[Tre89]{Treschev89}
D.~V. Treshch\"ev.
\newblock A mechanism for the destruction of resonance tori in {H}amiltonian
  systems.
\newblock {\em Mat. Sb.}, 180(10):1325--1346, 1439, 1989.

\bibitem[WZ03]{WilczakZ03}
D.~Wilczak and P.~Zgliczy{\'n}ski.
\newblock Heteroclinic connections between periodic orbits in planar restricted
  circular three-body problem--a computer assisted proof.
\newblock {\em Communications in mathematical physics}, 234(1):37--75, 2003.

\bibitem[WZ05]{WilczakZ05}
D.~Wilczak and P.~Zgliczy{\'n}ski.
\newblock Heteroclinic connections between periodic orbits in planar restricted
  circular three body problem. part ii.
\newblock {\em Communications in mathematical physics}, 259:561--576, 2005.

\bibitem[Zeh76]{Zehnder76}
E.~Zehnder.
\newblock Generalized implicit function theorems with applications to some
  small divisor problems. {II}.
\newblock {\em Comm. Pure Appl. Math.}, 29(1):49--111, 1976.

\end{thebibliography}
